\documentclass[9pt,reqno]{amsart}
\usepackage[all]{xy}

\usepackage{amsmath,amsthm,amssymb,eucal, amscd, mathrsfs,amsfonts,pxfonts}

\usepackage{hyperref}
\usepackage{array}

 \def\commutatif{\ar@{}[rd]|{\circlearrowleft}}

\newcommand{\eq}[1][r]
   {\ar@<-3pt>@{-}[#1]
    \ar@<-1pt>@{}[#1]|<{}="gauche"
    \ar@<+0pt>@{}[#1]|-{}="milieu"
    \ar@<+1pt>@{}[#1]|>{}="droite"
    \ar@/^2pt/@{-}"gauche";"milieu"
    \ar@/_2pt/@{-}"milieu";"droite"}

\newtheorem{thm}{Theorem}[section]
\newtheorem{pro}[thm]{Proposition}
\newtheorem{lem}[thm]{Lemma}
\newtheorem{cor}[thm]{Corollary}

\newtheorem{rem}[thm]{Remark}

\newtheorem{nota}[thm]{Notations}

\newtheorem{df}[thm]{Definition}

\newtheorem{rmk}[thm]{Remark}

\newtheorem{ex}[thm]{Example}

\newcommand{\cA}{{\mathcal A}}
\newcommand{\cB}{{\mathcal B}}
\newcommand{\cC}{{\mathcal C}}

\newcommand{\cE}{{\mathcal E}}
\newcommand{\cF}{{\mathcal F}}
\newcommand{\cG}{{\mathcal G}}
\newcommand{\cH}{{\mathcal H}}

\newcommand{\cK}{{\mathcal K}}
\newcommand{\cL}{{\mathcal L}}
\newcommand{\cM}{{\mathcal M}}

\newcommand{\cX}{{\mathcal X}}

\newcommand{\CC}{{\mathbb C}}
\newcommand{\EE}{{\mathbb E}}

\newcommand{\NN}{{\mathbb N}}

\newcommand{\uc}{{\mathbb S}^1}
\newcommand{\TT}{{\mathbb T}}

\newcommand{\ZZ}{{\mathbb Z}}

\newcommand{\id}{{\mathbf 1}}

\newcommand{\Id}{{\text{id}}}

\newcommand{\wGa}{{\widetilde \Gamma}}

\newcommand{\Gpdo}{{\mathcal{G}^{(0)}}}
\newcommand{\grpd}{\xymatrix{\cG \dar[r]^r_s & \Gpdo}}
\newcommand{\gamgpd}{\xymatrix{\Ga \dar[r]^r_s & \Ga^{(0)}}}

\newcommand{\Gamo}{{\Gamma^{(0)}}}

\newcommand{\Mo}{{\mathcal{M}^{(0)}}}

\def\dar[#1]{\ar@<2pt>[#1]\ar@<-2pt>[#1]}
  \entrymodifiers={!!<0pt,0.7ex>+} 

\newcommand{\fm}{{\mathfrak m}}
\newcommand{\fr}{{\mathfrak r}}
\newcommand{\fs}{{\mathfrak s}}

\newcommand{\sfH}{{\mathsf H}}

\newcommand{\To}{\longrightarrow}

\newcommand{\mto}{\longmapsto}

\newcommand{\cstar}{C^{\ast}}
\newcommand{\Ga}{\Gamma}

\newcommand{\al}{\alpha}

\newcommand{\ph}{\phi}

\newcommand{\g}{\gamma}

\newcommand{\Lam}{\Lambda}

\newcommand{\Co}{{\mathcal C}_0}

\def\<{\langle}
\def\>{\rangle}
\let\ipscriptstyle=\scriptscriptstyle
\def\lipsqueeze{{\mskip -3.0mu}}
\def\ripsqueeze{{\mskip -3.0mu}}
\def\ipcomma{\nobreak\mathrel{,}\nobreak}
\newbox\ipstrutbox
\setbox\ipstrutbox=\hbox{\vrule height8.5pt
width 0pt}
\def\ipstrut{\copy\ipstrutbox}
\def\lip#1<#2,#3>{\mathopen{\relax_{\ipstrut\ipscriptstyle{
#1}}\lipsqueeze
\langle} #2\ipcomma #3 \rangle}
\def\blip#1<#2,#3>{\mathopen{\relax_{\ipstrut
\ipscriptstyle{ #1}}\lipsqueeze\bigl\langle} #2\ipcomma #3 \bigr\rangle}
\def\rip#1<#2,#3>{\langle #2\ipcomma #3
\rangle_{\ripsqueeze\ipstrut\ipscriptstyle{#1}}}
\def\brip#1<#2,#3>{\bigl\langle #2\ipcomma #3
\bigr\rangle_{\ripsqueeze\ipstrut\ipscriptstyle{#1}}}
\def\angsqueeze{\mskip -6mu}
\def\smangsqueeze{\mskip -3.7mu}
\def\trip#1<#2,#3>{\langle\smangsqueeze\langle #2\ipcomma #3
\rangle\smangsqueeze\rangle_{\ripsqueeze\ipstrut\ipscriptstyle{#1}}}
\def\btrip#1<#2,#3>{\bigl\langle\angsqueeze\bigl\langle #2\ipcomma
#3
\bigr\rangle
\angsqueeze\bigr\rangle_{\ripsqueeze\ipstrut\ipscriptstyle{#1}}}
\def\tlip#1<#2,#3>{\mathopen{\relax_{\ipstrut\ipscriptstyle{
#1}}\lipsqueeze \langle\smangsqueeze\langle} #2\ipcomma #3
\rangle\smangsqueeze\rangle}
\def\btlip#1<#2,#3>{\mathopen{\relax_{\ipstrut\ipscriptstyle{
#1}}\lipsqueeze
\bigl\langle\angsqueeze\bigl\langle} #2\ipcomma #3
\bigr\rangle\angsqueeze\bigr\rangle}

\def\ip(#1|#2){(#1\mid #2)}
\def\bip(#1|#2){\bigl(#1 \mid #2\bigr)}
\def\Bip(#1|#2){\Bigl( #1 \bigm| #2 \Bigr)}
\def\h[#1,#2]{[#1,#2]_{H}}

\newcommand{\mydot}{\mathbin{:}}
\newcommand{\Lip}{\lip \scriptstyle\star}
\newcommand{\Rip}{\rip \scriptstyle\star}
\newcommand{\tLip}{\tlip \scriptstyle\star}
\newcommand{\tRip}{\trip \scriptstyle\star}
\newcommand{\X}{\mathsf{X}}

\def\ipp(#1|#2){\ip({#1}|{#2})_{\pi}}

\title{Equivalence of Fell Systems and their Reduced $\cstar$-Algebras}

\author{El-ka\"ioum M. Moutuou}
\thanks{The first author is supported by the German Research Foundation (DFG) and the Universit\'e franco-allemande (DFH-UFA) via the \emph{International Research Training Group} 1133 "Geometry and Analysis of Symmetries"}
\address{Department of Mathematics\\
Paderborn University\\
Warburger Str. 100\\
D-33098 Paderborn, Germany, and
} 
\address{Universit\'e Paul Verlaine - Metz\\
LMAM - CNRS UMR 7122\\
B\^atiment A, Ile du Saulcy\\
57000 Metz, France}
\email{moutuou@math.upb.de}

\author{Jean-Louis Tu}
\address{Universit\'e Paul Verlaine - Metz\\
LMAM - CNRS UMR 7122\\
B\^atiment A, Ile du Saulcy\\
57000 Metz, France}
\email{tu@univ-metz.fr}

\begin{document}

\maketitle

\begin{abstract}
This paper is aimed at investigating links between Fell bundles over Morita equivalent groupoids and their corresponding reduced $\cstar$-algebras. Mainly, we review the notion of \emph{Fell pairs} over a Morita equivalence of groupoids, and give the analogue of the Renault's Equivalence Theorem for the reduced $\cstar$-algebras of \emph{equivalent} Fell systems. Eventually, we will use this theorem to connect the reduced $\cstar$-algebra of an $\uc$-central groupoid extension to that of its associated Dixmier-Douady bundle.
\end{abstract}

\section*{Introduction}
A \emph{Fell system} consists of a pair $(\cG,\cE)$, where $\cE$ is a Fell bundle over the groupoid $\cG$. The notion of (Morita) equivalence of Fell systems was first introduced by S. Yamagami in ~\cite{Yam}, and by then it was studied by Muhly in ~\cite{Muh} and very recently by Muhly and Williams in ~\cite{MW2} where the authors prove that if $(\Ga,\cF)$ and $(\cG,\cE)$ are equivalent, then their full $\cstar$-algebras $\cstar(\Ga;\cF)$ and $\cstar(\cG;\cE)$ are Morita equivalent (see ~\cite[Theorem 11]{Muh}, and ~\cite[Theorem 6.4]{MW2}). However, it has not been known so far whether an equivalence of Fell systems gives rise to a Morita equivalence between the associated reduced $\cstar$-algebras.    \

 The first motivation of our work came from twisted $K$-theory: to every groupoid $\cG$ and every cocycle $\alpha\in \check{C}^2(\cG_\bullet,\TT)$ is associated a Fell system $(\Ga_\al,L_\al)$, and the twisted $K$-groups $K_\alpha^\ast(\cG)$ are defined as the $\cstar$-algebraic $K$-groups of the reduced $\cstar$-algebra $\cstar_r(\Ga_\al,L_\al)$ (cf. ~\cite{TXL}). Moreover, it is known that when $\al\sim \beta$, then not only $\Ga_\al$ is Morita equivalent to $\Ga_\beta$ but also the associated reduced $\cstar$-algebras $\cstar_r(\Ga_\al,L_\al)$ and $\cstar_r(\Ga_\beta,L_\beta)$ are Morita equivalent (see ~\cite[Proposition 3.3]{TXL}); so that $K_\al^\ast(\cG)\cong K_\beta^\ast(\cG)$. This has led us to a generalisation of the so-called Renault's equivalence Theorem for reduced groupoid $\cstar$-algebras (~\cite[Theorem 13]{SW}) to Fell systems.\
 
 We recall from ~\cite{TXL} some concepts related to groupoids such as \emph{generalized homomorphisms} and Dixmier-Douady bundles in \S1, and we review the basics of Fell systems and their reduced $\cstar$-algebras from ~\cite{Kum} and ~\cite{TXL} in \S2. In \S3, we discuss the notion of equivalence of Fell systems of  ~\cite{Muh} and ~\cite{MW2} from another formalism that better suits with the construction of the linking Fell systems introduced in \S4. The equivalence theorem for the reduced $\cstar$-algebras of Fell systems is proved in \S5, and then, in \S6, we apply this theorem to link the $\cstar$-algebra associated to an $\uc$-central extension of a groupoid $\cG$ to the reduced cross-product $\cA\rtimes_r\cG$, where $\cA$ is some Dixmier-Douady bundle.

\section{Preliminaries}

Although we assume that the reader is familiar with the language of groupoids (see for instance ~\cite{Ren}), we recall some of their basics used substantially throughout this paper. All the groupoids we are working with are supposed to be Hausdorff, locally compact, second countable, and are equipped with Haar systems. They are also assumed to have finite-dimentional base spaces, in the sense of ~\cite{DD}. 

\subsection{}Given two groupoids $\grpd$ and $\gamgpd$, a \emph{generalized morphism} $Z:\Ga\To \cG$ consists of a (locally compact Hausdorff) space $Z$, two maps $\xymatrix{\Gamo & Z \ar[l]_{\fr} \ar[r]^{\fs} & \Gpdo}$ ($\fs$ and $\fr$ are called \emph{generalized source map} and \emph{generalized range map} of $Z$, respectively), a left action of $\Ga$ on $Z$ with respect to $\fr$, a right action of $\cG$ on $Z$ with respect to $\fs$, such that the two actions commute, and $Z\To \Gamo$ is a right $\cG$-principal bundle. Such a morphism is a \emph{Morita equivalence} if in addition, $Z\To \Gpdo$ is a left $\Ga$-principal bundle; in this case, we say that $\Ga$ and $\cG$ are \emph{Morita equivalent}, and we write $\Ga\sim_Z \cG$. The terminology of \emph{generalized morphism} is justified by the fact that any \emph{strict groupoid homomorphism} (see ~\cite{Ren}) $f:\Ga\To \cG$ induces a generalized one $Z_f:\Ga\To \cG$, where $Z_f:=\Gamo\times_{f,\Gpdo,r}\cG$, with generalized source and range $\fs(y,g):=s(g)$ and $\fr(y,g):=y$, while the actions are $\g\cdot(s(\g),g):=(r(\g),f(\g)g)$, and $(y,g)\cdot g':=(y,gg')$.
\subsection{}If $\Ga\sim_Z\cG$, then $\cG\sim_{Z^{-1}}\Ga$, where $Z^{-1}$ is $Z$ as topological space, and if $\flat:Z\To Z^{-1}$ is the identity map, the generalized source and range are $\fs^\flat(\flat(z)):=\fr(z)$ and $\fr^\flat(\flat(z)):=\fs(z)$. The left $\cG$-action on $Z^{-1}$ is given by $g\cdot\flat(z):=\flat(zg^{-1})$ for $(z,g^{-1})\in Z\ast\cG$, and the right $\Ga$-action is $\flat(z)\cdot\g:=\flat(\g^{-1}z)$ whenever $(\g^{-1},z)\in \Ga\ast Z$. If $\Ga\sim_{Z_1}\cG_1\sim_{Z_2}\cG$, then $\Ga\sim_{Z_1\times_{\cG_1}Z_2}\cG$, where $Z_1\times_{\cG_1}Z_2$ is the quotient of the fibre product space $Z_1\times_{\fs_1,\cG_1^{(0)},\fr_2}Z_2$ by the equivalence relation $(z_1,z_2)\sim (z_1g_1,g_1^{-1}z_2)$.

\subsection{}A \emph{Dixmier-Douady bundle} $\cA$ over $\cG$ is a locally trivial bundle $\cA\To \Gpdo$ with fibre the $\cstar$-algebra $\cK$ of compact operators on the separable infinite-dimensional Hilbert space $\cH=l^2(\NN)$, together with an action $\alpha$ by automorphisms of $\cG$ on $\cA$; that is, a continuous family of isomorphisms of $\cstar$-algebras $\alpha_g:\cA_{s(g)}\To \cA_{r(g)}$ such that $\alpha_{gh}=\alpha_g\circ \alpha_h$ whenever $g$ and $h$ are composable, and $\alpha_{g^{-1}}=\alpha_g^{-1}$. Such a bundle is represented by the triple $(\cA,\cG,\al)$. Let $\cH_\cG:=L^2(\cG)\otimes \cH$, where $L^2(\cG)$ is the $\cG$-equivariant $\Co(\Gpdo)$-Hilbert module obtained by completing $\cC_c(\cG)$ with respect to the scalar product $\<\xi,\eta\>(x)=\int_{\cG^x}\overline{\xi(g)}\eta(g)d\mu_\cG^x(g)$. We say that two Dixmier-Douady bundles $\cA$ and $\cB$ are \emph{Morita equivalent}, and write $\cA\sim \cB$, if $\cA\otimes\cK(\cH_\cG)\cong \cB\otimes\cK(\cH_\cG)$. The set of Morita equivalence classes of Dixmier-Douady bundles forms an abelian group $\mathbf{Br}(\cG)$ called \emph{the Brauer group} of $\cG$. We refer to ~\cite{KMRW}, ~\cite{Tu}, or ~\cite{M} for more details about the structures of $\mathbf{Br}(\cG)$. 


\section{Fell systems and their reduced $\cstar$-algebras}

If $p:\cE\To \cG$ is a Banach bundle, we set \[\cE^{[2]}:=\left\{(e_1,e_2)\in \cE\times \cE \mid (p(e_1),p(e_2))\in \cG^{(2)}\right\}.\] Let $\fm: \cG^{(2)}\To \cG$ denote the partial multiplication of the groupoid $\grpd$. Then $\fm^\ast \cE\To \cG^{(2)}$ is a Banach bundle. 

\begin{df}(cf. \cite{Kum}, ~\cite[Appendix A]{TXL})~\label{df:fell_system}.
A \emph{multiplication} on $\cE$ consists of a continuous map $\cE^{[2]}\ni (e_1,e_2)\mto \left((p(e_1),p(e_2)),e_1e_2\right)\in \fm^\ast\cE $ satisfying the following properties:
\begin{itemize}
        \item[(i)] the induced map $\cE_g\times \cE_h\To \cE_{gh}$ is bilinear for all $(g,h)\in \cG^{(2)}$;
        \item[(ii)] (\textbf{associativity}) $(e_1e_2)e_3=e_1(e_2e_3)$ whenever the multiplication makes sense; and
        \item[(iii)] $\|e_1e_2\| \leq \|e_1\|\|e_2\|$, for every $(e_1,e_2)\in \cE^{[2]}$.		
	\end{itemize}
A ${}^\ast$-\emph{involution} on $\cE$ is a continuous $2$-periodic map ${}^\ast: \cE\ni e\mto e^\ast\in\cE$ such that 
\begin{itemize}
		\item[(iv)] $p(e^\ast)=p(e)^{-1}$, and 
		\item[(v)] for all $g\in \cG$, the induced map ${}^\ast:\cE_g\To \cE_{g^{-1}}$ is conjugate linear.	
		\end{itemize}		
Finally, we say that $p:\cE\To \cG$ is a Fell bundle if in addition the following conditions hold:
\begin{itemize}
      \item[(vi)] $(e_1e_2)^\ast=e_2^\ast e_1^\ast, \forall (e_1,e_2)\in \cE^{[2]}$;
      \item[(vii)] $\|e^\ast e\|=\|e\|^2, \forall e\in \cE$; in particular, $\cE_x$ is a $\cstar$-algebra, for $x\in \Gpdo$;
      \item[(viii)] $e^\ast e\ge 0, \forall e\in \cE$; and
      \item[(ix)] (\textbf{fullness}) the image of the map $\cE_g\times \cE_h\To \cE_{gh}$ spans a dense subspace of $\cE_{gh}$, for all $(g,h)\in \cG^{(2)}$.			
		\end{itemize}
If we are given such a Fell bundle, we say that $(\cG,\cE)$ is a \emph{Fell system}. 				
\end{df}

\begin{ex}
If $\cA$ is a Dixmier-Douady bundle over $\cG$ with action $\al$, we get a Fell system $(\cG,s^\ast\cA)$, where $s^\ast\cA$ is the $\cstar$-bundle over $\cG$ obtained by pulling back $\cA$ through the source map $s:\cG\To \Gpdo$. The multiplication on $s^\ast\cA$ is given by $\cA_{s(g)}\times\cA_{s(h)}\ni (a,b)\mto \al_{h^{-1}}(a)b\in \cA_{s(gh)}=\cA_{s(h)}$, and the ${}^\ast$-involution is $\cA_{s(g)}\ni a\mto \al_g(a)^\ast \in \cA_{s(g^{-1})}=\cA_{r(g)}$.	
\end{ex}

Given a Fell system $(\cG,\cE)$, we turn the space $\cC_c(\cG;\cE)$ of compactly supported continuous sections of $\cE$ into a convolution algebra by setting 
\begin{eqnarray}
	(\xi\ast \eta)(g):=\int_{\cG^{r(g)}}\xi(\g)\eta(\g^{-1}g)d\mu_\cG^{r(g)}(\g), \ \text{and} \ \xi^\ast(g):=\xi(g^{-1})^\ast, \ \text{for} \ \xi,\eta\in \cC_c(\cG;\cE), \ g\in \cG.
\end{eqnarray}

Let $\|\xi\|_1:=\text{sup}_{x\in \Gpdo}\int_{\cG^x}\|\xi(g)\|d\mu_\cG^x(g)$. We next define the \emph{I-norm} $\|\cdot\|_I$ by $\|\xi\|_I:=\text{max}\{\|\xi\|_1,\|\xi^\ast\|_1\}$. Then, the completion $L^1(\cG;\cE)$ of $\cC_c(\cG;\cE)$ with respect to $\|\cdot\|_I$ is a Banach ${}^\ast$-algebra. Its envelopping $\cstar$-algebra $\cstar(\cG;\cE)$ is called the \emph{full $\cstar$-algebra of} $(\cG,\cE)$.\\

Note that we have a $\cstar$-bundle over the base $\Gpdo$, defined as the pull-back of $\cE$ along the identity map $\Gpdo\hookrightarrow \cG$; we denote it by $\cE^{(0)}$. We can view $\cE^{(0)}$ as the restriction $\cE_{|\Gpdo}$, once we have identified $\Gpdo$ with a subset of $\cG$. Moreover, equipped with the pointwise norm, $A:=\Co(\Gpdo;\cE^{(0)})$ is a $\cstar$-algebra. We will usually write $A_x$ for the $\cstar$-algebra which is the fibre of $\cE^{(0)}$ over $x\in \Gpdo$. \\
The following proposition is proved, for instance, in ~\cite{M}, (and in ~\cite{Kum} in the case of proper groupoids), so we omit the proof.

\begin{pro}
$\cC_c(\cG;\cE)$ is a pre-Hilbert (left) $A$-module under the operations
\begin{eqnarray}
		(f\cdot\xi)(g):=f(r(g))\xi(g), \ \text{for} \ f\in A, \xi\in \cC_c(\cG;\cE), \ a\in \cG, \ \text{and}			
	\end{eqnarray}	
\begin{eqnarray}
{}_A\<\xi,\eta\>(x):=\int_{\cG^x}\xi(g)\eta(g)^\ast d\mu_\cG^x(g), \ \text{for} \ \xi,\eta\in \cC_c(\cG;\cE), \ x\in \Gpdo.		
\end{eqnarray}
\end{pro}

Let $L^2(\cG;\cE)$ be the Hilbert $A$-module obtained by completing $\cC_c(\cG;\cE)$ with respect to the norm 
\[\|\xi\|_2:=\|{}_A\<\xi,\xi\>\|^{1/2}, \ \text{for} \ \xi\in \cC_c(\cG;\cE). \]
Then, left multiplication by an element of $\cC_c(\cG;\cE)$ (i.e. the map $\pi_l(\xi):\eta\mto \xi\ast\eta, \ \xi,\eta\in \cC_c(\cG;\cE)$) is a bounded $A$-linear operator with respect to the norm $\|\cdot\|_2$ which is adjointable (see ~\cite{M}). Hence $\pi_l$ extends to a ${}^\ast$-monorphism 
\[\pi_l:\cC_c(\cG;\cE)\To \cL(L^2(\cG;\cE)):=\cL_A(L^2(\cG;\cE)).\]
The extension of $\pi_l$ to $L^1(\cG;\cE)$ is known as \emph{the left regular representation} of $L^1(\cG;\cE)$.

\begin{df}(cf. ~\cite[A.3]{TXL}).
Under the above notations, the closure of the image $\pi_l(\cC_c(\cG;\cE))$ in $\cL(L^2(\cG;\cE))$ with respect to the operator norm is called the \emph{reduced $\cstar$-algebra} of the Fell system $(\cG,\cE)$, and is denoted by $\cstar_r(\cG;\cE)$; i.e.
\[\cstar_r(\cG;\cE):=\overline{\pi_l(\cC_c(\cG;\cE))}\subset \cL(L^2(\cG;\cE))).\]	
\end{df}

\begin{rem}
One can think of $\cstar_r(\cG;\cE)$ as the completion of the convolution ${}^\ast$-algebra $\cC_c(\cG;\cE)$ with respect to the \emph{reduced norm} $\|\cdot\|_r$ given by $\|\xi\|_r:=\text{sup}\{\|\pi_l(\xi)\eta\|_2 \mid \eta\in \cC_c(\cG;\cE), \|\eta\|_2\leq 1 \}$.	
\end{rem}

\begin{rem}
If $(\cA,\cG,\al)$ is a Dixmier-Douady bundle, then the reduced $\cstar$-algebra $\cstar_r(\cG;s^\ast\cA)$ associated to the Fell bundle $s^\ast\cA$, denoted by $\cA\rtimes_r\cG$, is called the \texttt{reduced crossed product} of $(\cA,\cG,\al)$; it plays an important role in twisted $K$-theory of groupoids (~\cite{TXL}, ~\cite{Tu}).	
\end{rem}

\medskip

Alternatively, we will sometimes use another definition of the reduced norm, which is a generalisation of that of ~\cite{SW}. Suppose we are given a right Fell system $(\cG,\cE)$. Then, for all $x\in \Gpdo$, consider the inclusion $i_x:\cG_x\To \cG$. Then, as in ~\cite[A.3]{TXL}, we define the (right) Hilbert $A_x$-module $L^2(\cG_x;\cE)$ as the completion of $\cC_c(\cG_x;i_x^\ast \cE)$ with respect to the inner product $\<\xi,\eta\>_{A_x}:=\int_{\cG_x}\xi(g)^\ast\eta(g)d\mu_x(g)$ (the right action being $(\xi\cdot a):\cG_x\ni g\mto \xi(g)\cdot a\in \cE_g$). The following lemma is very easy to prove.

\medskip

\begin{lem}~\label{lem:pi_x}
Let $(\cG,\cE)$ be as above. Then, for all $x\in \Gpdo$, left multiplication by elements of $\cC_c(\cG;\cE)$ gives a ${}^\ast$-representation $\pi_x^{\cG}:\cC_c(\cG;\cE)\To \cL_{A_x}(L^2(\cG_x;\cE))$. Moreover, we have 
\[\|\xi\|_{\cstar(\cG;\cE)}:=\|\xi\|_r=\underset{x\in \Gpdo}{\text{sup}} \{\|\pi_x^{\cG}(\xi)\|, \ \forall \xi\in \cC_c(\cG;\cE)\] 	
 \end{lem} 

\section{Equivalence of Fell systems}

In this section, we are presenting the notion of equivalences of Fell bundles over Morita equivalent groupoids. Our definitions are slight modifications of those given by P. Muhly and D. Williams in  ~\cite[\S.6]{MW2}. \
  
Suppose that $Z$ is a (right) principal $\cG$-space; that is, there is a principal $\cG$-action $\al:Z\ast\cG\To Z$. If $\pi:\cX\To Z$ is a Banach bundle, and if $p:\cE\To \cG$ is a Fell bundle, we set \[\cX\ast\cE:=\{(u,e)\in \cX\times \cE \mid (\pi(u),p(e))\in Z\ast\cG\}.\]
\begin{df}
A \emph{right Fell $\cG$-pair} over the principal $\cG$-space $Z$ is a pair $(\cX,\cE)$ consisting of a Fell bundle $\cE$ over $\cG$, a Banach bundle $\pi:\cX\To Z$, and a continuous map $\cX\ast\cE\ni (u,e)\mto ue\in \al^\ast\cX$, such that
\begin{itemize}
  		\item[(i)](\textbf{bilinearity}) for all $(z,g)\in Z\ast\cG$, the induced map $\cX_z\times \cE_g\To \cX_{zg}$ is bilinear, and is compatible with the scalar multiplication; i.e. $(\lambda u)e=u(\lambda e)=\lambda(ue), \forall \lambda\in \CC, (u,e)\in \cX_z\times \cE_g$;
  		\item[(ii)](\textbf{associativity}) if $(z,g)\in Z\ast\cG$ and $(g,h)\in \cG^{(2)}$, one has $u(e_1e_2)=(ue_1)e_2, \forall (u,e_1,e_2)\in \cX_z\times\cE_g\times\cE_h$;
  		\item[(iii)] $\|ue\|=\|u\|\|e\|, \forall (u,e)\in \cX_z\times \cE_g$;
  		\item[(iv)](\textbf{faithfullness}) the induced map $\cX_z\times \cX_g\To \cX_{zg}$ spans a dense subspace of $\cX_{zg}$.  
  	\end{itemize}  
 We also say that $(\cG,\cE)$ acts on $\cX$ on the right over $Z$. \\
 Likewise, one defines a \emph{left Fell $\cG$-pair} $(\cE,\cX)$ over a principal left $\cG$-space $Z$.	
\end{df}

\begin{rmk}
 Notice that if $(\cX,\cE)$ is a right Fell $\cG$-pair over $Z$, then for every $z\in Z$, $\cX_z$ is a right $\cE_{\fs(z)}$-module.
\end{rmk}

Now suppose that $\Ga\sim_Z\cG$. Then there are a continuous $\Ga$-valued inner product ${}_\Ga<\cdot,\cdot>:Z\times_{\Gpdo}Z^{-1}\To \Ga$, and a continuous $\cG$-valued inner product $<\cdot,\cdot>_{\cG}:Z^{-1}\times_{\Gamo}Z\To \cG$, defined as follows
\begin{itemize}
	\item for $(z,\flat(z'))\in Z\times_{\Gamo}Z^{-1}$, ${}_\Ga<z,z'>$ is the unique element of $\Ga$ such that $z={}_\Ga<z,z'>\cdot z'$;
	\item for $(\flat(z),z')\in Z^{-1}\times_{\Gamo}Z$, $<z,z'>_\cG$ is the unique element of $\cG$ such that $z'=z\cdot<z,z'>_\cG$.
\end{itemize}
Observe that these functions are well defined, for $Z\To \Gamo$ is a $\cG$-principal bundle, and $Z\To \Gpdo$ is a $\Ga$-principal bundle. Furthermore, they satisfy the following equalities (cf. ~\cite[\S.6.1]{Parav}):
\begin{align}
	{}_\Ga<z,z'>^{-1} &={}_\Ga<z',z>, \forall (z,\flat(z'))\in Z\times_{\Gpdo}Z^{-1},\\
	<z,z'>_\cG^{-1} &=<z',z>_\cG, \forall (\flat(z),z')\in Z^{-1}\times_{\Gamo}Z, \ \text{and} \\
	z\cdot<z',z">_\cG &={}_\Ga<z,z'>\cdot z", \ \forall (z,\flat(z'),z")\times Z\times_{\Gpdo}Z^{-1}\times_{\Gamo}Z.
\end{align}

\medskip

\begin{lem}~\label{lem:right_vs_left-action}
Let $\Ga\sim_Z\cG$. Then, any right Fell $\cG$-pair $(\cX,\cE)$ over $Z$ gives rise to the left Fell $\cG$-pair $(\cE,\overline{\cX})$ over the inverse $Z^{-1}$, where $\overline{\cX}$ is defined as the conjugate bundle of $\cX$. A similar statement holds for a left $\Ga$-pair over $Z$.	
\end{lem}

\begin{proof}
By definition $\overline{\cX}$ is $\cX$ as space. If $\flat:\cX\To \overline{\cX}$ denotes the identity map, we define the projection $\bar{\pi}:\overline{\cX}\To Z^{-1}$ by $\bar{\pi}(\flat(u)):=\flat(\pi(u))$. The fibre $\overline{\cX}_{\flat(z)}$ is the conjugate Banach space of $\cX_z$; the left $\cG$-action on $\overline{\cX}$ is $g\cdot\flat(u):=\flat(u\cdot g^{-1})$, while the left action of $\cE$ on $\overline{\cX}$ is given by $\cE_g\times\overline{\cX}_{\flat(z)}\ni (e,\flat(u))\mto \flat(u\cdot e^\ast)\in \overline{\cX}_{g\cdot \flat(u)}$.	
\end{proof}

Let us fix some notations that will be used in the sequel. Suppose $\Ga\sim_Z\cG$. If $(\cX,\cE)$ is a Fell $\cG$-pair, we define the topological spaces $\cX\ast\overline{\cX}:=\{(u,\flat(u'))\in \cX\times\overline{\cX}\mid (\pi(u),\bar{\pi}(\flat(u')))\in Z\times_{\Gpdo}Z^{-1}\}$ and $\overline{\cX}\ast\cX:=\{(\flat(u),u')\in \overline{\cX}\times \cX\mid (\bar{\pi}(\flat(u)),\pi(u'))\in Z^{-1}\times_{\Gamo}Z\}$. Observe that the space $Z\times_{\Gpdo}Z^{-1}$ is a locally compact groupoid with base $Z$ as follows: the product is $(z,\flat(z'))\cdot(z',\flat(z")):=(z,\flat(z"))$, the source of $(z,\flat(z'))$ is $z'$, its range is $z$, and its inverse is $(z',\flat(z))$. Similarly $Z^{-1}\times_{\Gamo}Z$ is a locally compact groupoid with base $Z^{-1}$.\\
 If $\Ga\sim_Z\cG$, and if $(\cG,\cE)$ and $(\Ga,\cF)$ are Fell systems, we denote by $\cE_{>_\cG}$ and $\cF_{{}_\Ga<}$ the Fell bundles over $Z^{-1}\times_{\Gamo}Z$ and $Z\times_{\Gpdo}Z^{-1}$, respectively, obtained by pulling back $\cE\To \cG$ along the continuous map $<\cdot,\cdot>_\cG$, and $\cF\To \Ga$ along the continuous map ${}_\Ga<\cdot,\cdot>$, respectively. Note that, for instance, the fiber of $\cE_{>_\cG}$ over $(\flat(z),z')$ is isomorphic to $\cE_{<z,z'>_\cG}$. 

\begin{df}
Assume $\Ga\sim_Z\cG$ and $(\cX,\cE)$ is a Fell $\cG$-pair over $Z$. An $\cE$-\emph{valued inner product} on $\cX$ is a continuous map $\<\cdot,\cdot\>_\cE:\overline{\cX}\ast\cX\To \cE_{>_\cG}, (\flat(u),u')\mto \<u,u'\>_\cE$, such that
\begin{itemize}
	\item[(i)] for $(\flat(z),z')\in Z^{-1}\times_{\Gamo}Z$, the induced map $\<\cdot,\cdot\>_\cE:\overline{\cX}_{\flat(z)}\times \cX_{z'}\To \cE_{<z,z'>_\cG}$ is linear in both the first and the second variable;
	\item[(ii)](\textbf{$\cE$-linearity}) if $(\flat(z),z')\in Z^{-1}\times_{\Gamo}Z$ and $(z,g)\in Z\ast\cG$, then $\<u,u'\>_\cE\cdot e=\<u,u'\cdot e\>_\cE, \forall (\flat(u),u',e)\in \overline{\cX}_{\flat(z)}\times\cX_{z'}\times \cE_g$;
	\item[(iii)] $\<u,u'\>_\cE^\ast=\<u',u\>_\cE\in \cE_{<z,z'>_\cG^{-1}}=\cE_{<z',z>_\cG}$;
	\item[(iv)](\textbf{positivity}) for all $z\in Z$ and $u\in \cX_z$, $\<u,u\>_\cE\ge 0$ in $\cE_{<z,z>_\cG}=\cE_{\fs(z)}$; and the equality $\<u,u\>_\cE=0$ implies $u=0$.  
\end{itemize}
In this case, we say that $\cX$ is a \emph{right $(\cG,\cE)$-inner product module} over $Z$.\\
Likewise, if $(\cF,\cX)$ is a left Fell $\Ga$-pair, one defines an $\cF$-valued inner product ${}_\cF\<\cdot,\cdot\>:\cX\ast\overline{\cX}\To \cF_{{}_\Ga<}$, all the actions being considered on the left. 
\end{df}

\begin{rmk}
Observe that conditions (ii) and (iii) of the definition imply that $\<u\cdot e,u'\>_\cE=e^\ast\cdot\<u,u'\>_\cE$, whenever the multiplications and the inner product are defined. Moreover, for all $z\in Z$, $\cX_z$ is a pre-Hilbert $A_{\fs(z)}$-module. 
\end{rmk}

\medskip 

\begin{df}~\label{df:equiv-Fell}
An \emph{equivalence} between two Fell systems $(\Ga,\cF)$ and $(\cG,\cE)$ is a pair $(Z,\cX)$ such that $\Ga\sim_Z\cG$, $\cX$ is a left $(\cF,\Ga)$-inner product module and a right $(\cG,\cE)$-inner product module over $Z$, with the following properties
\begin{itemize}
  		\item[(i)] (\textbf{equivariance}) for all $(\g,z,g)\in \Ga\ast Z\ast\cG$, the multiplication $\cF_\g\times\cX_z\times\cE_g\To \cX_{\g zg}$ is associative; i.e. $f\cdot(u\cdot e)=(f\cdot u)\cdot e, \forall (f,u,e)\in \cF_\g\times\cX_z\times \cE_g$;
  		\item[(ii)] (\textbf{compatibility}) for all $(z,\flat(z'),z")\in Z\times_{\Gpdo}Z^{-1}\times_{\Gamo}Z$ and $(u,\flat(u'),u")\in \cX_z\times\overline{\cX}_{\flat(z')}\times \cX_{z"}$, ${}_\cF\<u,u'\>\cdot u"=u\cdot \<u',u"\>_\cE$ in $\cX_{z\cdot <z',z">_\cG}=\cX_{{}_\Ga<z,z'>\cdot z"}$;
  		\item[(iii)] the $\cF$-valued inner product is full; i.e., the image of the induced map $\cX_z\times \overline{\cX}_{\flat(z')}\To \cF_{{}_\Ga<z,z'>}$ spans a dense subspace of $\cF_{{}_\Ga<z,z'>}$;
  		\item[(iv)] the $\cE$-valued inner product is full.
  	\end{itemize}
In this case, we write $(\Ga,\cF)\sim_{(Z,\cX)}(\cG,\cE)$.  	
\end{df}

\medskip

\begin{rem}~\label{rem:equiv-symmetry}
It follows from Definition~\ref{df:equiv-Fell} and Lemma~\ref{lem:right_vs_left-action} that if $(\Ga,\cF)\sim_{(Z,\cX)}(\cG,\cE)$, then $(\cG,\cE)\sim_{(Z^{-1},\overline{\cX})}(\Ga,\cF)$. Furthermore, it is starightforward that for all $z\in Z$, (the completion with respect to the inner products of) $\cX_z$ is an imprimitivity $B_{\fr(z)}$-$A_{\fs(z)}$-bimodule.
\end{rem}

\medskip 

\begin{ex}~\label{ex:reflexivity}
If $(\cG,\cE)$ is a Fell system, then $(\cG,\cE)\sim_{(Z_\cG,\cE)}(\cG,\cE)$, where $Z_\cG$ is the space of morphisms $\cG^{(1)}$ (with which we identify $\cG$). Indeed, $Z_\cG$ implements a Morita equivalence $\cG\sim_{Z_\cG} \cG$, the generalized source and range maps being the source and range maps $s$ and $r$ of $\grpd$, together with the canonical left and right actions given by partial multiplications. Notice that $Z_\cG^{-1}=\{g^{-1}\mid g\in \cG\}$. It is easy to see that the inner products ${}_\cG<\cdot,\cdot>:Z_\cG\times_{\Gpdo} Z_\cG^{-1}\To \cG$ and $<\cdot,\cdot>_\cG:Z_\cG^{-1}\times_{\Gpdo}Z_\cG\To \cG$ are ${}_\cG<g,h>=gh^{-1}$ and $<g,h>_\cG=g^{-1}h$, respectively. $\cE$ acts on itself over $Z_\cG$ by definition of a Fell bundle. Now, the conjugate bundle $\overline{\cE}\To Z_\cG^{-1}$ is given fibrewise by $\overline{\cE}_{g^{-1}}=\{e^\ast \mid e\in \cE_g\}$. The inner products are $\cE_g\times \overline{\cE}_{h^{-1}}\ni (e_1,e_2^\ast)\mto e_1e_2^\ast\in \cE_{gh^{-1}}$, and $\overline{\cE}_{g^{-1}}\times \cE_h\ni (e_1^\ast,e_2)\mto e_1^\ast e_2\in \cE_{g^{-1}h}$. It is straightforward that all the conditions of Definition~\ref{df:equiv-Fell} are satisfied.	
\end{ex}

By virtue of Remark~\ref{rem:equiv-symmetry} and Example~\ref{ex:reflexivity}, equivalence of Fell systems is symmetric and reflexive. Also, it is not hard to show that it is transitive, so that it defines an equivalence relation among the collection of Fell systems (cf. ~\cite{M}).\\
In the sequel, we will need the following result.

\medskip

\begin{pro}~\label{pro:full-Ga-G-bimodule}
If $(\Ga,\cF)\sim_{(Z,\cX)}(\cG,\cE)$, $\cC_c(Z;\cX)$ is a full pre-inner product $\cC_c(\Ga;\cF)$-$\cC_c(\cG;\cE)$-bimodule with respect to the inductive limit topologies~\footnote{Since we are dealing with Banach ${}^\ast$-algebras, the only properties we take into account here are the continuity of the actions and the pre-inner products with respect to the inductive limit topologies, the compatibility between the actions and the pre-inner products, and the fullness of the latters.} under the following operations:
\begin{align}
 	(\xi\cdot\phi)(z) & :=\int_{\Ga^{\fr(z)}}\xi(\g)\phi(\g^{-1}\cdot z)d\mu_{\Ga}^{\fr(z)}(\g),  \\
 	(\phi\cdot\eta)(z) & := \int_{\cG^{\fs(z)}}\phi(z\cdot g)\eta(g^{-1})d\mu_\cG^{\fs(z)}(g) ,
 	\end{align} 

\begin{align}
	{}_{\cC_c(\Ga;\cF)}\langle \phi,\psi\rangle(\g) & := \int_{\cG^{\fs(z)}}{}_\cF \left\langle \phi(z\cdot g),\psi(\g^{-1}\cdot z\cdot g)\right\rangle d\mu_\cG^{\fs(z)}(g), \ \text{where} \ \fr(z)=r(\g), \ \text{and}
\end{align}
 
\begin{align}	
	\langle \phi,\psi\rangle_{\cC_c(\cG;\cE)}(g) & := \int_{\Ga^{\fr(z)}}\left\langle \phi(\g^{-1}\cdot z),\psi(\g^{-1}\cdot z\cdot g)\right\rangle_{\cE}d\mu_\Ga^{\fr(z)}(\g), \ \text{where} \ \fs(z)=r(g),	
\end{align}	
\end{pro}

\begin{proof}
See ~\cite{MW2} or ~\cite{M}.	
\end{proof}

\medskip
We will adopt the following notations.

\medskip

\begin{nota}~\label{nota:inner-prod}
1. For the sake of simplicity, we will sometimes write $\Lip<\cdot,\cdot>$ for ${}_{\cC_c(\Ga;\cF)}\<\cdot,\cdot\>$ and $\Rip<\cdot,\cdot>$ for $\<\cdot,\cdot\>_{\cC_c(\cG;\cE)}$. \\
2. As in ~\cite{SW}, if $\xi\in \cC_c(\cG;\cE), \eta\in \cC_c(\Ga;\cF)$, and $\phi,\psi\in \cC_c(Z^{-1};\overline{\cX})$, we will write $\xi\mydot \phi$ and $\phi\mydot \eta$ for the left and right actions of $\cC_c(\cG;\cE)$ and $\cC_c(\Ga;\cF)$ on $\cC_c(Z^{-1};\overline{\cX})$, respectively, and we will write $\tLip<\phi,\psi>$ for ${}_{\cC_c(\cG;\cE)}\<\phi,\psi\>$ and $\tRip<\phi,\psi>$ for $\<\phi,\psi\>_{\cC_c(\Ga;\cF)}$. 
\end{nota}

\medskip

\begin{rem}~\label{rem:approx-identity}
We should note that the proof of Proposition~\ref{pro:full-Ga-G-bimodule} is mostly based on the crucial result proved in ~\cite[Proposition 6.10]{MW2} that guarantees the existence of a net $\{f_\lambda\}_{\lambda\in \Lam}$ in $\cC_c(\Ga;\cF)$ of the form $f_\lambda=\sum_{i=1}^{n_\lambda}\Lip<\ph_i^\lambda,\ph_i^\lambda>$, with each $\ph_i^\lambda\in \cC_c(Z;\cX)$, which is an approximate identity with respect to the inductive limit topology for both the left action of $\cC_c(\Ga;\cF)$ on itself and on $\cC_c(Z;\cX)$. By symmetry, a similar statement holds for $(\cG,\cE)$. In particular, by Example~\ref{ex:reflexivity}, the same result  shows that for any Fell system $(\cG,\cE)$, $\cC_c(\cG;\cE)$ admits an approximate identity for the inductive limit topology.
\end{rem}

\section{The linking Fell system}

In this section, we use some constructions from ~\cite[Chapter 6]{Parav} and ~\cite[\S.2]{MRW1}. 
If $\Ga\sim_Z\cG$, then form the \emph{Linking groupoid} $\xymatrix{\cM \dar[r] &\cM^{(0)}}$ by setting: $\cM:=\Ga\sqcup Z\sqcup Z^{-1}\sqcup \cG$, and $\cM^{(0)}:=\Gamo\sqcup\Gpdo$, the source and range maps $s_\cM$ and $r_\cM$ being the obvious ones. The partial multiplication of $\cM$ is given by
\begin{align*} 
\cM^{(2)}\To \cM,  \left\{ \begin{array}{llll}
 (\g_1,\g_2) & \in \Ga^{(2)}: & \g_1\g_2 & \in \Ga\\
 (\g,z) & \in \Ga\ast Z: & \g.z & \in Z\\
 (z,\flat(z')) & \in Z\times_{\Gpdo}Z^{-1}: & z.\flat(z'):={}_\Ga\langle z,z'\rangle & \in \Ga \\
 (z,g) & \in Z\ast\cG: & z.g & \in Z \\
 (\flat(z),z') & \in Z^{-1}\times_{\Gamo}Z: & \flat(z).z':=\langle z,z'\rangle_\cG & \in \cG\\
 (\flat(z),\g) & \in Z^{-1}\ast\Ga: & \flat(z).\g:=\flat(\g^{-1}.z) & \in Z^{-1}\\
 (g,\flat(z)) & \in \cG\ast Z^{-1}: & g.\flat(z):=\flat(zg^{-1}) & \in Z^{-1}\\
(g_1,g_2) & \in \cG^{(2)}: & g_1g_2 & \in \cG
 	\end{array}\right\},	
 \end{align*} 
so that $\cM^{(2)}=\Ga^{(2)}\sqcup \Ga\ast Z\sqcup Z\times_{\Gpdo}Z^{-1}\sqcup Z\ast\cG \sqcup Z^{-1}\ast\Ga \sqcup Z^{-1}\times_{\Gamo}Z \sqcup \cG\ast Z^{-1}\sqcup \cG^{(2)}$. Finally, the inversion in $\cM$ is defined by
\begin{align*}
\cM\To \cM, \left\{ \begin{array}{llll}
	\Ga & \ni \g & \mto \g^{-1} & \in \Ga \\
	Z & \ni z & \mto \flat(z) & \in Z^{-1}\\
	Z^{-1} & \ni \flat(z) & \mto z & \in Z \\
	\cG & \ni g & \mto g^{-1} & \in \cG
          \end{array}\right\}.	
\end{align*}
With these structures, $\xymatrix{\cM \dar[r] & \cM^{(0)}}$ is a locally compact Hausdorff groupoid with open source and range maps ( ~\cite[Proposition 6.2.2]{Parav}). \\
Now, let $\mu_\Ga$ and $\mu_\cG$ be left Haar systems on $\Ga$ and $\cG$, respectively. Then, if $\Ga\sim_Z\cG$, there exists a full $\fr$-system~\footnote{See for instance ~\cite{Ren1} for the definition.} $\mu_Z=\{\mu_Z^y\}_{y\in \Gamo}$ of Radon measures on $Z$ determined by 
\begin{eqnarray}
		\mu_Z^y(\phi):=\int_{\cG^{\fs(z)}}\phi(z\cdot g)d\mu_\cG^{\fs(z)}(g),
	\end{eqnarray}
for all $y\in \Gamo$ and $\phi\in \cC_c(Z)$, where $z$ is some arbitrary element of the fibre $Z_y=\fr^{-1}(y)$. Furthermore, $\mu_Z$ is a left Haar system on $Z$ for the left action of $\Ga$; that is, for all $\g\in \Ga$ and $\phi\in \cC_c(Z)$, we have $	\int_{Z_{r(\g)}}\phi(z)d\mu_Z^{r(\g)}(z)=\int_{Z_{s(\g)}}\phi(\g.z)d\mu_Z^{s(\g)}(z)$ (see ~\cite[\S.6.4]{Parav}, ~\cite{SW}). Similarly, considering the inverse $Z^{-1}:\cG\To \Ga$, the Haar system $\mu_\Ga$ induces a left Haar system $\mu_{Z^{-1}}=\{\mu_{Z^{-1}}^x\}_{x\in\Gpdo}$ on $Z^{-1}$ for left action of $\cG$. Note that we have $\text{supp}\mu_{Z^{-1}}^x=(\fr^\flat)^{-1}(x)=Z^{-1}_x$, and that for $\phi\in \cC_c(Z^{-1})$ and $\flat(z)\in Z^{-1}_x$, we have
\begin{eqnarray}
	\mu_{Z^{-1}}^x(\phi):=\int_{\Ga^{\fs^\flat(\flat(z))}=\Ga^{\fr(z)}}\phi(\flat(\g^{-1}z))d\mu_\Ga^{\fr(z)}(\g).	
	\end{eqnarray}	

Moreover, $\mu_\Ga,\mu_\cG, \mu_Z$, and $\mu_{Z^{-1}}$ induces a left Haar system $\mu_\cM$ on $\cM$ as it is shown in the following proposition.

\medskip

\begin{pro} 
Assume $\Ga\sim_Z\cG$, and $\mu_\Ga$ and $\mu_\cG$ are left Haar systems on $\Ga$ and $\cG$, respectively. Then, under the above constructions, there is a left Haar system $\mu_\cM=\{\mu_\cM^\omega\}_{\omega \in \Mo}$ on the linking groupoid $\cM$ determined by 
\begin{eqnarray}
	\mu_\cM^\omega(F):= \left\{ \begin{array}{ll}
	\mu_\Ga^\omega(F_{|\Ga})+\mu_Z^{\omega}(F_{|Z}), & \text{if}  \ \omega \in \Gamo, \ \text{and} \\
	\mu_{Z^{-1}}^\omega(F_{|Z^{-1}})+\mu_\cG^\omega(F_{|\cG}), & \text{if}  \ \omega \in \Gpdo,		
	\end{array} 
	\right.	
	\end{eqnarray}	
for all $\omega\in \Mo$ and $F\in \cC_c(\cM)$. 	
\end{pro}
\begin{proof}
See ~\cite[Proposition 6.4.5]{Parav}, or ~\cite[Lemma 4]{SW}.	
\end{proof}

\medskip

\begin{pro}
Suppose $(\Ga,\cF)\sim_{(Z,\cX)}(\cG,\cE)$. Then we define a Banach bundle $\cL$ over the linking groupoid $\cM$, where  $\cL$ is the topological space $\cL:=\cF\sqcup \cX \sqcup \overline{\cX}\sqcup\cE$, the projection is given by
 \begin{eqnarray}
 	p^{\cL}: \cL\To \cM, \ \left\{ \begin{array}{llll}
 		\cF & \ni f & \mto p^\cF(f) & \in \Ga\\
 		\cX & \ni u & \mto \pi(u) & \in Z \\
 		\overline{\cX} & \ni \flat(v) & \mto \flat(\pi(v)) & \in Z^{-1} \\
 		\cE & \ni e & \mto p^\cE(e) & \in \cG
 	\end{array}
 	\right\}.	
 	\end{eqnarray}
 Moreover, $p^\cL:\cL\To \cM$ is a Fell bundle with respect to the multiplication $\cL^{[2]}\To \fm^\ast\cL$ and involution $({}^\ast):\cL\To \cL$ respectively given by 
 \begin{eqnarray}~\label{eq:mult_Linking_Fell}
  \left\{ \begin{array}{llll}
 \cF_{\g_1}\times\cF_{\g_2} & \ni (f_1,f_2) & \mto f_1f_2 & \in \cF_{\g_1\g_2}, \ \text{for} \ (\g_1,g_2)\in \Ga^{(2)}\\
 \cF_\g\times \cX_z & \ni (f,u) & \mto f\cdot u & \in \cX_{\g\cdot z},\ \text{for} \ (\g,z)\in \Ga\ast Z \\
 \cX_{z_1}\times \overline{\cX}_{\flat(z_2)} & \ni (u,\flat(v)) & \mto {}_\cF\langle u,v\rangle & \in \cF_{{}_\Ga<z_1,z_2>},\ \text{for} \ (z_1,\flat(z_2))\in Z \times_{\Gpdo}Z^{-1} \\
 \cX_z\times \cE_g & \ni (u,e) & \mto u\cdot e & \in \cX_{zg}, \ \text{for} \ (z,g)\in Z\ast\cG \\
 \overline{\cX}_{\flat(z)}\times \cF_\g & \ni (\flat(u),f) & \mto \flat(f^\ast\cdot u) & \in \overline{\cX}_{\flat(\g^{-1}z)},\ \text{for} \ (\flat(z),\g)\times Z^{-1}\ast\Ga \\
  \overline{\cX}_{\flat(z_1)}\times \cX_{z_2} & \ni (\flat(u),v) & \mto \langle u,v\rangle_\cE & \in \cE_{<z_1,z_2>_\cG}, \ \text{for} \ (\flat(z_1),z_2)\in Z^{-1}\times_{\Gamo}Z\\
  \cE_g\times \overline{\cX}_{\flat(z)} & \ni (e,\flat(u)) & \mto \flat(u\cdot e^\ast) & \in \overline{\cX}_{\flat(zg^{-1})},\ \text{for} \ (g,\flat(z))\in \cG\ast Z^{-1}\\
  \cE_g\times \cE_h & \ni (e_1,e_2) & \mto e_1e_2 & \in \cE_{gh}, \ \text{for} \ (g,h)\in \cG^{(2)}	
 \end{array}  \right\}			
 		\end{eqnarray}	
and 

\begin{eqnarray}~\label{eq:star_Linking_Fell}
 ({}^\ast):\cL\to \cL, \ \left\{ \begin{array}{llll}
 \cF_\g & \ni f & \mto f^\ast & \in \cF_{\g^{-1}}, \ \text{for} \ \g\in \Ga \\
 \cX_z & \ni u & \mto \flat(u) & \in \overline{\cX}_{\flat(z)}, \ \text{for} \ z\in Z\\
 \overline{\cX}_{\flat(z)} & \ni \flat(v) & \mto v & \in \cX_z, \ \text{for} \ \flat(z) \in Z^{-1}\\
 \cE_g & \ni e & \mto e^\ast & \in \cE_{g^{-1}}, \ \text{for} \ g\in \cG \end{array}\right\}.				
 			\end{eqnarray} 
 $\cL$ is called the \texttt{linking Fell bundle}, and $(\cM,\cL)$ is the \texttt{linking Fell system}.							
\end{pro}

\begin{proof}
It is clear that $p^\cL:\cL\To \cM$ is a Banach bundle. Next, observe that all of the conditions of Definition~\ref{df:fell_system} are verified by the operations~\eqref{eq:mult_Linking_Fell} and~\eqref{eq:star_Linking_Fell} by merely applying Definition~\ref{df:equiv-Fell} to the equivalences $(Z,\cX)$ and $(Z^{-1},\overline{\cX})$.	
\end{proof}

At this point, we can do integration on $\cM$ with values on the the linking Fell bundle $\cL$. We then can form the convolution ${}^\ast$-algebra $\cC_c(\cM;\cL)$. Note that we have an isomorphism of convolution ${}^\ast$-algebras $\cC_c(\cM;\cL)\cong \cC_c(\Ga;\cF)\oplus \cC_c(Z;\cX)\oplus \cC_c(Z^{-1};\overline{\cX})\oplus \cC_c(\cG;\cE)$; so that an element $\xi\in \cC_c(\cM;\cL)$ can be written as a matrix 
\[\xi= \left( \begin{array}{ll}
            \xi_{11} & \xi_{12} \\
            \xi_{21} & \xi_{22}
	
\end{array}   \right),\]
where $\xi_{11}:=\xi_{|\Ga}\in \cC_c(\Ga;\cF), \ \xi_{12}:=\xi_{|Z}\in \cC_c(Z;\cX),\ \xi_{21}:=\xi_{|Z^{-1}}\in \cC_c(Z^{-1};\overline{\cX})$, and $\xi_{22}:=\xi_{|\cG}\in \cC_c(\cG;\cE)$. With respect to this decomposition, the involution in $\cC_c(\cM;\cL)$ is given by
\[ \xi^\ast= \left( \begin{array}{ll}
\xi_{11}^\ast & \xi_{21}^\ast \\
\xi_{12}^\ast & \xi_{22}^\ast
	
\end{array}
\right) = \left( \begin{array}{ll}
\xi_{11}^\ast & \flat \circ \xi_{21}\circ \flat \\
\flat\circ \xi_{12}\circ \flat & \xi_{22}^\ast 	
\end{array}
\right),
\] 
where $\xi_{11}^\ast$ and $\xi_{22}^\ast$ are the images of $\xi_{11}$ and $\xi_{22}$ under the standard involutions in $\cC_c(\Ga;\cF)$ and $\cC_c(\cG;\cE)$, respectively. Furthermore, routine calculations (cf. ~\cite{M}) show that the convolution in $\cC_c(\cM;\cL)$ is given by
\begin{eqnarray}
 \left(\begin{array}{ll}
 \xi_{11} &\xi_{12} \\ & \\ \xi_{21} & \xi_{22}
 \end{array}  \right)	\ast \left( \begin{array}{ll}
 \eta_{11} & \eta_{12}\\ & \\ \eta_{21} & \eta_{22}
 \end{array}
 \right)  = \left( \begin{array}{ll}
 \xi_{11}\ast\eta_{11} + \Lip<\xi_{12},\eta_{21}^\ast> & \xi_{11}\cdot \eta_{12} + \xi_{12}\cdot \eta_{22} \\
 & \\
 (\eta^\ast_{11}\cdot \xi_{21}^\ast)^\ast+(\eta_{21}^\ast\cdot \xi_{22}^\ast)^\ast & \Rip<\xi_{21}^\ast,\eta_{12}> +\xi_{22}\ast\eta_{22}
 \end{array}  \right).	
 	\end{eqnarray}
 
\medskip 
 	 		
Suppose $(\Ga,\cF)\sim_{(Z,\cX)}(\cG,\cE)$. For $x\in \Gpdo$, we also denote by $\cX\To Z_x$ the pull-back of $\cX\To Z$ along the inclusion $Z_x\hookrightarrow Z$. Then, $L^2(Z_x;\cX)$ is the completion of $\cC_c(Z_x;\cX)$ with respect to the $A_x$-valued inner product $\Rip<\phi,\psi>(x)=\int_{\Ga^{\fr(z)}}\<\phi(\g^{-1}\cdot z),\psi(\g^{-1}\cdot z)\>_{\cE}d\mu_\Ga^{\fr(z)}(\g)$, where $\fs(z)=x$, and the right $A_x$-action $(\phi\cdot a)(z):=\phi(z)a$, for $\phi\in \cC_c(Z_x;\cX), a\in A_x$. Thus, $L^2(Z_x;\cX)$ is a Hilbert $A_x$-module. Similarly, for all $y\in \Gamo$, one can form the Hilbert $B_y$-module $L^2(Z^{-1}_y;\overline{\cX})$.\\
The following proposition will be crucial in the proof of the equivalence theorem (Theorem~\ref{thm:Renault_equiv}).

\begin{pro}~\label{pro:R_x}
Suppose $(\Ga,\cF)\sim_{(Z,\cX)}(\cG,\cE)$. For $x\in \Gpdo$, the left action of $\cC_c(\Ga;\cF)$ on $\cC_c(Z_x;\cX)$ induces a ${}^\ast$-representation $R_x^\Ga:\cC_c(\Ga;\cF)\To \cL_{A_x}(L^2(Z_x;\cX))$ that factors through the $\cstar$-algebra $\cstar_r(\Ga;\cF)$. Similarly, for all $y\in \Gamo$, we get a representation $R_y^{\cG}:\cstar_r(\cG;\cE)\To \cL_{B_y}(L^2(Z^{-1}_y;\overline{\cX}))$.					
\end{pro} 

\begin{proof}
 Let $\xi\in \cC_c(\Ga;\cF)$; then for $\phi,\psi\in \cC_c(Z_x;\cX)$, simple calculations give $\Rip<\xi\cdot \phi,\psi>(x)=\Rip<\phi,\xi^\ast \cdot \psi>(x)$. It follows that the $A_x$-linear operator $\cC_c(Z_x;\cX)\ni \phi\mto \xi\cdot \phi\in \cC_c(Z_x;\cX)$ is adjointable, and then bounded with respect to the norm $\|\cdot\|_{L^2(Z_x;\cX)}$, which gives the ${}^\ast$-representation $R_x^\Ga:\cC_c(\Ga;\cF)\To \cL_{A_x}(L^2(Z_x;\cX)), \xi\To (R_x^\Ga(\xi):\phi\mto \xi\cdot\phi)$.\\
 Now, let $z_0\in Z_x$, and let $y:=\fr(z_0)$. Then, to complete the proof it suffices to check that for all $\xi\in \cC_c(\Ga;\cF)$, $\|R_x^\Ga(\xi)\|\le \|\pi_y^\Ga(\xi)\|$, where $\pi_y^\Ga:\cC_c(\Ga;\cF)\To \cL_{B_y}(L^2(\Ga_y;\cF))$ is the representation defined in Lemma~\ref{lem:pi_x}.\\
Consider the (left) Hilbert $B_y$-module $\cX_{z_0}$, and form the interior tensor product $L^2(\Ga_y;\cF)\otimes_{B_y}\cX_{z_0}$ which is a right Hilbert $A_x$-module under the operations defined on simple tensors by: $(\xi\otimes u)\cdot a:=\xi\otimes (ua)$, and $\<\xi\otimes u,\eta\otimes v\>:=\<u,\<\xi,\eta\>_{B_y}\cdot v\>_{A_x}$. Then, the map
\begin{align}~\label{eq:density:FX--X}
  	 u_{z_0}:L^2(\Ga_y;\cF)\otimes_{B_y}\cX_{z_0} \To L^2(Z_x;\cX), \ \sum_i \xi_i\otimes u_i \mto \sum_i\xi_i\cdot u_i, \end{align}
 where for $\xi\in \cC_c(\Ga_y;\cF)$ and $u\in \cX_{z_0}$, $(\xi\cdot u)(z):=\xi({}_\Ga<z,z_0>)\cdot u\in \cX_{{}_\Ga<z,z_0>\cdot z_0}$, is an isomorphism of Hilbert $A_x$-modules. The map~\eqref{eq:density:FX--X} is clearly $A_x$-linear and injective. To see that it is surjective, first notice that the well defined map $Z_x\ni z\mto {}_\Ga<z,z_0>\in \Ga_y,$ is a homeomorphism of $\Ga$-spaces (its inverse being $\Ga_y\ni \g\mto \g\cdot z_0\in Z_x$). Next, for all $z\in Z_x$, the linear span of the image of $\cF_{{}_\Ga<z,z_0>}\times \cX_{z_0}\ni (f,u)\mto f\cdot u\in \cX_{{}_\Ga<z,z_0>\cdot z_0}$ is dense in $\cX_{{}_\Ga<z,z_0>\cdot z_0}$ by definition of a Fell pair; so that, using the Weierstrass theorem, $$\text{span}\left\{\eta\cdot u: Z_x\ni z\mto \eta({}_\Ga<z,z_0>)\cdot u\in\cX_{{}_\Ga<z,z_0>\cdot z_0} \mid \eta\in \cC_c(\Ga_y;\cF), u\in \cX_{z_0}\right\}$$ is dense in $\cC_c(Z_x;\cX)$ in the inductive limit topology. It follows that any $\phi\in \cC_c(Z_x;\cX)$ is the inductive limit of some $\sum_i\eta_i\cdot u_i=u_{z_0}(\sum_i\eta_i\otimes u_i)$. We then have an isomorphism of $\cstar$-algebras $$\tilde{u}_{z_0}:\cL_{A_x}(L^2(\Ga_y;\cF)\otimes_{B_y}\cX_{z_0})\To \cL_{A_x}(L^2(Z_x;\cX))$$ such that $\tilde{u}_{z_0}(T)\left(\sum_i\xi_i\cdot u_i\right):=u_{z_0}\left(T\left(\sum_i\xi_i\otimes u_i\right)\right)$, for all $T\in \cL_{A_x}(L^2(\Ga_y;\cF)\otimes_{B_y}\cX_{z_0})$. Furthemore, the following diagram is commutative 
 \[
 \xymatrix{\cC_c(\Ga;\cF)\ar[d]_{\pi_y^\Ga} \ar[r]^{R_x^\Ga} & \cL_{A_x}(L^2(Z_x;\cX)) \\ 
 \cL_{B_y}(L^2(\Ga_y;\cF)) \ar[r] & \cL_{A_x}(L^2(\Ga_y;\cF)\otimes_{B_y}\cX_{z_0})\ar[u]_{\tilde{u}_{z_0}} }\]
 where the lower horizontal arrow is the map $T\mto T\otimes \Id$ (cf. for instance ~\cite[p.50]{Lan}). Indeed, let $\xi\in \cC_c(\Ga;\cF)$, and $\phi\in \cC_c(Z_x;\cX)$. Without loss of generality, we can suppose that $\phi=\eta\cdot u$; then, $$\tilde{u}_{z_0}(\pi_y^\Ga(\xi)\otimes \Id)\phi=(\pi_y^\Ga(\xi)\otimes\Id)(\eta\otimes u)=(\pi_y^\Ga(\xi)\eta)\otimes u=(\xi\ast\eta)\cdot u=\xi\cdot (\eta\cdot u)=R_x^\Ga(\xi)(\eta\otimes u)=R_x^\Ga(\xi)\phi,$$ which completes the proof since $u_{z_0}$ is an isomorphism and $\|\pi_y^\Ga(\xi)\otimes\Id\|\le \|\pi_y^\Ga(\xi)\|$ (see ~\cite[p.50]{Lan}).						
\end{proof}   	 					

\section{The equivalence theorem for reduced $\cstar$-algebras of Fell systems}

We start this section by the following observations. Let $(\cG,\cE)$ be a Fell system and let $A:=\Co(\Gpdo;\cE^{(0)})$ be as usual. Suppose we are given a bounded continuous section $f\in \cC_b(\Gpdo;\cE^{(0)})$. Then, for $\xi\in \cC_c(\cG;\cE)$, we define an element $L_f\xi=:f\xi\in \cC_c(\cG;\cE)$ by setting: 
\begin{eqnarray}~\label{eq:L_f}
	L_f\xi(g):=f(r(g))\xi(g)\in \cE_g, \ \text{for all} \ g\in \cG.
\end{eqnarray}
Also, we define an element $\xi f\in \cC_c(\cG;\cE)$ by 
\begin{eqnarray}
	\cG\ni g\mto \xi f(g):=\xi(g)f(s(g))\in \cE_g.
\end{eqnarray}
Notice that $\cC_b(\Gpdo;\cE_{|\Gpdo})$ is a $\cstar$-algebra under pointwise operations and the supremum norm (~\cite[Lemma 3.2]{APT}).

\medskip

\begin{lem}
For all $f\in \cC_b(\Gpdo;\cE^{(0)})$, we have $L_f\in \cL(L^2(\cG;\cE))$, where $L_f$ is the element defined by~\eqref{eq:L_f}. Moreover, the map $L:\cC_b(\Gpdo;\cE^{(0)})\ni f \mto L_f\in \cL_A(L^2(\cG;\cE))$ is a  ${}^\ast$-homomorphism.
\end{lem}

\begin{proof}
$L_f$ is clearly continuous; also it is bounded  since $f$ is a bounded section (it is straightforward that $\|L_f\|_{op}\leq \|f\|$, where $\|\cdot\|_{op}$ is the operator norm in $\cL(L^2(\cG;\cE))$). If $\xi,\eta\in \cC_c(\cG;\cE)$ and $x\in \Gpdo$, then
\begin{align*}
	{}_A\<L_f\xi,\eta\>(x) & = \int_{\cG^x} L_f\xi(g^{-1})^\ast \eta(g^{-1})d\mu_\cG^x(g) \\
	& = \int_{\cG^x} \xi^\ast(g)f(r(g^{-1}))^\ast \eta(g^{-1})d\mu_\cG^x(g)\\
	& = {}_A\<\xi,L_{f^\ast}\eta\>(x);	
	\end{align*}
hence, $L_f$ is adjointable with adjoint $L_f^\ast:=L_{f^\ast}$. Moreover, $L_{f_1f_2}(\xi)=L_{f_1}(L_{f_2}(\xi)), \forall \xi\in \cC_c(\cG;\cE)$; thus $L_{f_1f_2}=L_{f_1}L_{f_2}, \forall f_1, f_2\in \cC_b(\Gpdo;\cE^{(0)})$.
\end{proof}

\medskip

\begin{pro}~\label{pro:C_b_vs_M(C_r)}
Let $(\cG,\cE)$ be as above. Then, $\cC_b(\Gpdo;\cE^{(0)})$ is a $\cstar$-subalgebra of $M(\cstar_r(\cG;\cE))$.	
\end{pro}

\begin{proof}
If $\pi_l(\xi)\in \cstar_r(\cG;\cE)$ and $f\in \cC_b(\Gpdo;\cE^{(0)})$, we put 
\begin{eqnarray}
	L_f(\pi_l\xi):=\pi_l(L_f\xi)=\pi_l(f\xi), \ \text{and} \ R_f(\pi_l\xi):=\pi_l(\xi f).
	\end{eqnarray}
We verify that with these formulas, we obtain a double centralizer $(L_f,R_f)\in M(\cstar_r(\cG;\cE))$. To see this, observe that for $\xi,\eta \in\cC_c(\cG;\cE)$ and $g\in \cG$, one has
\begin{align*}
	(f(\xi\ast\eta))(g) & = f(r(g))\int_{\cG^{r(g)}}\xi(h)\eta(h^{-1}g)d\mu_\cG^{r(g)}(h)\\
	& = \int_{\cG^{r(g)}}f(r(h))\xi(h)\eta(h^{-1}g)d\mu_\cG^{r(g)}(h)\\
	& = \int_{\cG^{r(g)}}(f\xi)(h)\eta(h^{-1}g)d\mu_\cG^{r(g)}(h)\\
	&= f\xi\ast \eta;	
	\end{align*}
and similarly one shows that $(\xi\ast\eta)f=\xi\ast\eta f$. Moreover, we have 		
\begin{align*}
(\xi f\ast\eta)(g) & = \int_{\cG^{r(g)}}\xi(h)f(s(h))\eta(h^{-1}g)d\mu_\cG^{r(g)}(h)\\
& = \int_{\cG^{r(g)}}\xi(h)f(r(h^{-1}g))\eta(h^{-1}g)d\mu_\cG^{r(g)}(h)\\
& = \int_{\cG^{r(g)}}\xi(h)(f\eta)(h^{-1}g)d\mu_\cG^{r(g)}(h)\\
& = (\xi\ast f\eta)(g);		
	\end{align*}
so that $R_f(\pi_l(\xi))\pi_l(\eta)=\pi_l(\xi)L_f(\pi_l(\eta))$, and by continuity, for every $f\in \cC_b(\Gpdo;\cE^{(0)})$, the pair $(L_f,R_f)$ verifies $R_f(a)b=aL_f(b)$ for all $a,b\in \cstar_r(\cG;\cE)$; i.e. $(L_f,R_f)\in M(\cstar_r(\cG;\cE))$.	
\end{proof}

In what follows, we identify the double centralizer $(L_f,R_f)$, and hence the element $f\in \cC_b(\Gpdo;\cE^{(0)})$,  with $L_f\in \cL(L^2(\cG;\cE))$, by considering $L_f$ as a \emph{multiplier} of $\cstar_r(\cG;\cE)$ under the formulas: $L_f\pi_l(\xi):=\pi_l(L_f\xi)=\pi_l(f\xi)$, and $\pi_l(\xi)L_f:=R_f(\pi_l(\xi))=\pi_l(\xi f)$.\\
Now, consider the field of $\cstar$-algebras \[M(\cE):=\coprod_{x\in \Gpdo}M(\cE_x)\]
over $\Gpdo$. Then, denote by $\cC_b^{str}(\Gpdo;M(\cE))$ the unital $\cstar$-algebra (under pointwise operations and the supremum norm) consisting of all the \emph{bounded strictly continuous sections} of $M(\cE)$ over $\Gpdo$ (see ~\cite[p.7]{APT} for details). Note that the unit $\id \in \cC_b^{str}(\Gpdo;M(\cE))$ is the section given by $\id:\Gpdo \ni x\mto (\Id_{\cE_x},\Id_{\cE_x})\in M(\cE_x)$, where $\Id_{\cE_x}:\cE_x\To \cE_x$ is the identity map. From Proposition~\ref{pro:C_b_vs_M(C_r)} we obtain the following corollary.

\medskip

\begin{cor}~\label{cor:C^str-M(C(E))}
Let $(\cG,\cE)$ be as above. Then $\cC_b^{str}(\Gpdo;M(\cE))$ is a unital $\cstar$-subalgebra of $M(\cstar_r(\cG;\cE))$.	
\end{cor}

\begin{proof}
The map $\cC_b(\Gpdo;\cE_{\Gpdo})\ni f\mto L_f\in M(\cstar_r(\cG;\cE))$ is non-degenerate; indeed, by considering the left Fell $\cG$-pair $(\cE,\cE)$ determined by the full maps $\cE_g\times \cE_h\To \cE_{gh}$, we see that for $f\in \cC_b(\Gpdo;\cE^{(0)})\subset \Co(\Gpdo;\cE^{(0)})$ and $\xi\in \cC_c(\cG;\cE)$, the element $L_f\xi\in \cC_c(\cG;\cE)$ is nothing but the canonical action of $\Co(\Gpdo;\cE^{(0)})$ on $\cC_c(\cG;\cE)$ defined by the formula $(f\cdot \xi)(g):=f(r(g))\xi(g)$. It follows that if $\{a_i\}_{i\in I}$ is an approximate identity of $\cC_b(\Gpdo;\cE^{(0)})$, then thanks to ~\cite[Lemma 6.1]{MW1}, for all $\xi\in \cC_c(\cG;\cE)$, $a_i\cdot\xi\To \xi$ in $\cC_c(\cG;\cE)$ with respect to the inductive limit topology. Thus $L_{a_i}\pi_l(\xi)=\pi_l(a_i\cdot\xi)\To \pi_l(\xi)$ in $\cstar_r(\cG;\cE)$. Whence, $L(\cC_b(\Gpdo;\cE^{(0)}))\cstar_r(\cG;\cE)$ is dense in $\cstar_r(\cG;\cE)$. Now, from  ~\cite[\S.3.12.10 and \S.3.12.12]{Ped79}, the map $L$ extends to a unital strictly continuous ${}^\ast$-homomorphism $M(\cC_b(\Gpdo;\cE^{(0)}))\To M(\cstar_r(\cG;\cE))$; this map is again denoted by $L$. Furthermore, from ~\cite[Lemma 3.1]{APT}, we have that $M(\Co(\Gpdo;\cE^{(0)}))=\cC_b^{str}(\Gpdo;M(\cE))$, which settles the result.	
\end{proof}

\medskip

\begin{pro}~\label{pro:p_Ga-p_G-compl}
Suppose $(\Ga,\cF)\sim_{(Z,\cX)}(\cG,\cE)$. Let $\chi_{_{\Gamo}}$ and $\chi_{_{\Gpdo}}$ be the characteristic functions of $\Gamo$ and $\Gpdo$ respectively. Then we get two elements $\chi_{_{\Gamo}}\id$ and $\chi_{_{\Gpdo}} \id$ of $\cC_b^{str}(\Mo;M(\cL))$, where $\id \in \cC_b^{str}(\Mo;M(\cL))$, defined by scalar multiplication. \\
Now define $$p_\Ga:=L_{\chi_{_{\Gamo}}\id}, \ \text{and} \ \  p_\cG:=L_{\chi_{_{\Gpdo}}\id} \in M(\cstar_r(\cM;\cL)).$$	Then $p_\Ga$ and $p_\cG$ are \texttt{complementary full projections}~\footnote{Recall from ~\cite{Bro} that a projection $p\in M(A)$ is said to be \emph{full} if $pAp$ is not contained in any proper closed two-sided ideal of $A$; that is, $\text{span}\{ApA\}$ is dense in $A$ (see for instance ~\cite{BGR} or ~\cite[p.50]{RW}). In this case, we say that $pAp$ is a \emph{full corner} of $A$. Two projections $p,q\in M(A)$ are \emph{complementary} if $p+q=1$, in which case $pAq$ is a $pAp$-$qAq$-imprimitivity bimodule; i.e. $pAp$ and $qAq$ are Morita equivalent. Conversely, two $\cstar$-algebras $A$ and $B$ are Morita equivalent if and only if there is a $\cstar$-algebra $C$ with complementary full corners isomorphic to $A$ and $B$, respectively (cf. ~\cite[Theorem 1.1]{BGR}, ~\cite[Theorem 3.19]{RW}).} in $M(\cstar_r(\cM;\cL))$.
\end{pro}

\begin{proof}
It is straightforward that $\chi_{_{\Gamo}}\id$ and $\chi_{_{\Gpdo}}\id$ are complementary projections of $\cC_b^{str}(\Mo;M(\cE))$. Hence, their images $p_\Ga$ and $p_\cG$ are complementary projections of $M(\cstar_r(\cM;\cL))$, by virtue of Corollary~\ref{cor:C^str-M(C(E))}.\

 Now, let $\xi, \eta\in \cC_c(\cM;\cL)$. Then 
\begin{align*}
\pi_l^\cM(\xi)p_\Ga\pi_l^\cM(\eta) & = \pi_l^\cM(\xi\ast p_\Ga \eta) =\pi_l^\cM(\xi p_\Ga\ast\eta)\\
& = \pi_l^\cM \left( \begin{array}{ll} 
\xi_{11}\ast\eta_{11} & \xi_{11}\cdot\eta_{12} \\ & \\
\xi_{21}\mydot\eta_{11} & \Rip<\xi_{21}^\ast,\eta_{12}>	
\end{array}  \right).	
\end{align*}
So, to check that $p_\Ga$ is full, we just have to show that 
\begin{align}
\text{span}\left\{ \pi_l^\cM \left( \begin{array}{ll} 
\xi_{11}\ast\eta_{11} & \xi_{11}\cdot\eta_{12} \\
\xi_{21}\mydot\eta_{11} & \Rip<\xi_{21}^\ast,\eta_{12}>	
\end{array}  \right) \mid \xi_{11}\in \cC_c(\Ga;\cF), \xi_{21}\in \cC_c(Z^{-1};\overline{\cX}), \right. \nonumber \\ 
\ \ \ \ \ \ \ \ \  \left. \eta_{12}\in \cC_c(Z;\cX), \eta_{11}\in \cC_c(\Ga;\cE)
\right\}	
\end{align}
is dense in $\cstar_r(\cM;\cL)$. But this is not hard to verify, by using the previous results. Indeed, the existence of an approximate identity in $\cC_c(\Ga;\cF)$ for both the left actions of $\cC_c(\Ga;\cF)$ on itself and on $\cC_c(Z;\cX)$ shows that elements of the form $\xi_{11}\ast\eta_{11}$, for $\xi_{11},\eta_{11}\in \cC_c(\Ga;\cF)$ span a dense subspace of $\cC_c(\Ga;\cF)$ and that elements of the form $\xi_{11}\cdot\eta_{12}$, for $\eta_{12}\in \cC_c(Z;\cX)$, span a dense subspace of $\cC_c(Z;\cX)$. Also, that elements of the form $\xi_{21}\mydot\eta_{11}$, where $\xi_{21}\in \cC_c(Z^{-1};\overline{\cX}), \eta_{11}\in \cC_c(\Ga;\cF)$, span a dense subspace of $\cC_c(Z^{-1};\overline{\cX})$ follows from the existence of an approximate identity in $\cC_c(\Ga;\cF)$ for the right action of $\cC_c(\Ga;\cF)$ on $\cC_c(Z^{-1};\overline{\cX})$ (cf. Remark~\ref{rem:approx-identity}). Finally, since $\cC_c(Z;\cX)$ is a full pre-inner product $\cC_c(\Ga;\cF)$-$\cC_c(\cG;\cE)$-bimodule (Proposition~\ref{pro:full-Ga-G-bimodule}), the image of $\Rip<\cdot,\cdot>$ is a dense subspace of $\cC_c(\cG;\cE)$. We then have shown that $\cstar_r(\cM;\cL)p_\Ga\cstar_r(\cM;\cL)$ is dense in $\cstar_r(\cM;\cL)$. In a similar fashion, we get that $\cstar_r(\cM;\cL)p_\cG\cstar_r(\cM;\cL)$ is dense in $\cstar_r(\cM;\cL)$, which completes the proof.
\end{proof}

\medskip

\begin{thm}~\label{thm:Renault_equiv}
Let $\Ga$ and $\cG$ be locally compact Hausdorff groupoids. Suppose $(\Ga,\cF)\sim_{(Z,\cX)}(\cG,\cE)$. Then the isomorphisms of convolution ${}^\ast$-algebras
\begin{eqnarray}~\label{eq:isom-conv_Ga--pMp}
	\cC_c(\Ga;\cF)\ni \xi_{11}\mto \left(\begin{array}{ll}\xi_{11} & 0 \\ 0 & 0\end{array}\right)\in p_\Ga\cC_c(\cM;\cL)p_\Ga,
\end{eqnarray}
and 
\begin{eqnarray}~\label{eq:isom-conv_G--pMp}
	\cC_c(\cG;\cE)\ni \eta_{22} \mto \left(\begin{array}{ll}0 & 0 \\ 0 & \eta_{22}  \end{array}  \right) \in p_\cG\cC_c(\cM;\cL)p_\cG	
	\end{eqnarray}
extend to two isomorphisms of $\cstar$-algebras 
\begin{eqnarray}~\label{eq:isom-cstar_GaG--pMp}
	\cstar_r(\Ga;\cF)\To p_\Ga\cstar_r(\cM;\cL)p_\Ga , \ \text{and} \ \ \cstar_r(\cG;\cE)\To p_\cG\cstar_r(\cM;\cL)p_\cG.
\end{eqnarray}	
In particular, $\cstar_r(\Ga;\cF)$ and $\cstar_r(\cG;\cE)$ are Morita equivalent with imprimitivity bimodule $p_\Ga\cstar_r(\cM;\cL)p_\cG$ which is isometrically isomorphic to the completion $\X_r$ of $\cC_c(Z;\cX)$ in the norm $$\|\phi\|_\cE:=\|\Rip<\phi,\phi>\|^{1/2}_{\cstar_r(\cG;\cE)}, \ \text{for} \ \phi\in \cC_c(Z;\cX).$$			
\end{thm}

\medskip

\begin{proof}
That the maps defined by~\eqref{eq:isom-conv_Ga--pMp} and~\eqref{eq:isom-conv_G--pMp} are isomorphisms of convolutions ${}^\ast$-algebras is obvious. \\
As previously, let us put $B:=\Co(\Gamo;\cF^{(0)})$ and $A:=\Co(\Gpdo;\cE^{(0)})$. Then \[\Co(\Mo;\cL^{(0)})\cong B\oplus A,\] as $\cstar$-algebras. Now, with respect to this decomposition, simple calculations show that 
\begin{eqnarray}
{}_{B\oplus A}\<\xi,\eta\>=\left({}_B\<\xi_{11},\eta_{11}\>+\Lip<\xi_{21}^\ast,\eta_{21}^\ast>_{|\Gamo}\right)\oplus \left(\Rip<\xi_{12},\eta_{12}>{}_{|\Gpdo}+{}_A\<\xi_{22},\eta_{22}\> \right),\end{eqnarray}
for all $\xi=\left(\begin{array}{ll}\xi_{11}&\xi_{12}\\ \xi_{21}&\xi_{22}\end{array}\right)$ and $\eta=\left(\begin{array}{ll}\eta_{11}&\eta_{12}\\ \eta_{21}&\eta_{22}\end{array}\right)$ in $\cC_c(\cM;\cL)$. In particular, suppose that $\xi=\left(\begin{array}{ll}\xi_{11}&0\\ 0&0\end{array}\right)\in p_\Ga\cC_c(\cM;\cL)p_\Ga$, then 
\[{}_{B\oplus A}\left\<\left(\begin{array}{ll}\xi_{11}&0\\ 0 &0\end{array}\right), \left(\begin{array}{ll}\xi_{11}&0\\ 0 &0\end{array}\right) \right\> ={}_B\<\xi_{11},\xi_{11}\>\oplus 0, \] 
so that 
\begin{eqnarray}
\left\|\left(\begin{array}{ll}\xi_{11}&0\\ 0 &0\end{array}\right)  \right\|_{L^2(\cM;\cL)}=\|\xi_{11}\|_{L^2(\Ga;\cF)};
\end{eqnarray}  
thus,~\eqref{eq:isom-conv_Ga--pMp} extends to an isometric $B$-linear map $u_\Ga$ of $B$-modules $$u_\Ga: L^2(\Ga;\cF)\To p_\Ga L^2(\cM;\cL)p_\Ga,$$ where $p_\Ga L^2(\cM;\cL)p_\Ga$ is the completion of $p_\Ga \cC_c(\cM;\cL)p_\Ga$ with respect to the norm of $L^2(\cM;\cL)$. 
Similarly, for $\xi_{22}\in \cC_c(\cG;\cE)$, we get 
\[{}_{B\oplus A}\left\<\left(\begin{array}{ll}0&0\\ 0 &\xi_{22}\end{array}\right), \left(\begin{array}{ll}0&0\\ 0 &\xi_{22}\end{array}\right) \right\> =0\oplus {}_A\<\xi_{22},\xi_{22}\>,  \]
and hence 
\begin{eqnarray}
 	\left\|\left(\begin{array}{ll}0&0\\ 0 &\xi_{22}\end{array}\right) \right\|_{L^2(\cM;\cL)}=\|\xi_{22}\|_{L^2(\cG;\cE)};
 \end{eqnarray} 
so that ~\eqref{eq:isom-conv_G--pMp} extends to an isometric $A$-linear map $u_\cG$ of $A$-modules $$u_\cG:L^2(\cG;\cE)\To p_\cG L^2(\cM;\cL)p_\cG.$$
Furthermore, since $u_\Ga$ and $u_\cG$ are surjective, then from ~\cite[Theorem 3.5]{Lan}, they are unitaries in $\cL_B(L^2(\Ga;\cF),p_\Ga L^2(\cM;\cL)p_\Ga)$ and $\cL_A(L^2(\cG;\cE),p_\cG L^2(\cM;\cL)p_\cG)$, respectively; in other words, $$L^2(\Ga;\cF) \approx p_\Ga L^2(\cM;\cL)p_\Ga$$ as Hilbert $B$-modules, and $$L^2(\cG;\cE)\approx p_\cG L^2(\cM;\cL)p_\cG$$ as Hilbert $A$-modules, here the sign "$\approx$" stands for \emph{unitarily equivalent}.
Moreover, it is very easy to see that the following diagrams commute:
\begin{eqnarray}
	\xymatrix{\cC_c(\Ga;\cF) \ar[d]^{\pi_l^\Ga} \ar[r]^\cong & p_\Ga \cC_c(\cM;\cL)p_\Ga \ar[d]^{\pi_l^\cM}\\
	\cL\left(L^2(\Ga;\cF)\right) \ar[r]^{\cong} & \cL\left(p_\Ga L^2(\cM;\cL)p_\Ga\right) } 
\ \ \ \ \ \  \ \ \ 	\xymatrix{\cC_c(\cG;\cE) \ar[d]_{\pi_l^\cG} \ar[r]^\cong & p_\cG\cC_C(\cM;\cL)p_\cG \ar[d]^{\pi_l^\cM}\\
	\cL\left(L^2(\cG;\cE)\right) \ar[r]^{\cong} & \cL\left(p_\cG L^2(\cM;\cL)p_\cG\right) }
\end{eqnarray}
It then only remains to check that for $\xi=\left(\begin{array}{ll}\xi_{11}&0\\0&0\end{array} \right)$ and $\eta=\left(\begin{array}{ll}0&0\\0&\eta_{22}\end{array}\right)$, we have $\|\xi\|_{\cstar_r(\cM;\cL)}=\|\xi_{11}\|_{\cstar_r(\Ga;\cF)}$ and $\|\eta\|_{\cstar_r(\cM;\cL)}=\|\eta_{22}\|_{\cstar_r(\cG;\cE)}$ which will lead to the desired isomorphisms of $\cstar$-algebras~\eqref{eq:isom-cstar_GaG--pMp} since $p_\Ga$ and $p_\cG$ are complementary (cf. Proposition~\ref{pro:p_Ga-p_G-compl}). However, by symmetry it suffices to check one of the latter equalities. To this end, we will use the constructions of Lemma~\ref{lem:pi_x}. \\
Note that we have \[\cC_c(\cM_\omega;\cL)=\left\{ \begin{array}{ll}\cC_c(\Ga_y;\cF)\oplus \cC_c(Z^{-1}_y;\overline{\cX}), & \rm{if} \ \omega=y\in \Gamo; \\
\cC_c(Z_x;\cX)\oplus \cC_c(\cG_x;\cE), & \rm{if} \ \omega=x\in \Gpdo\end{array}\right.\]
In other words, elements of $\cC_c(\cM_y;\cL)$, for $y\in \Gamo$ , are of the form $\left(\begin{array}{ll}\eta_{11} &0\\\eta_{21}&0\end{array}\right)$ with $\eta_{11}\in \cC_c(\Ga_y;\cF)$ and $\eta_{21}\in \cC_c(Z^{-1}_y;\overline{\cX})$, while elements of $\cC_c(\cM_x;\cL)$, for $x\in \Gpdo$, are of the form $\left(\begin{array}{ll}0&\eta_{12}\\0&\eta_{22}\end{array}\right)$ with $\eta_{12}\in \cC_c(Z_x;\cX)$ and $\eta_{22}\in \cC_c(\cG_x;\cE)$. Then, for all $y\in \Gamo$, and $\eta,\zeta\in \cC_c(\cM_y;\cL)$, one has 
\begin{align*}
\<\eta,\zeta\>_{B_y} = & \int_{\Ga_y}\eta_{11}(\g)^\ast \zeta_{11}(\g)^\ast d(\mu_\Ga)_y(\g) + \int_{Z^{-1}_y}{}_\cF\<\eta_{21}(\flat(z)),\zeta_{21}(\flat(z))\>d(\mu_{Z^{-1}})_y(\flat(z)),
\end{align*}
where $(\mu_{Z^{-1}})_y$ is the Radon measure on $Z^{-1}$ with support $Z_y^{-1}$, which is the image of $\mu^y$ on $Z$ under the "inversion" $Z^{-1}\To Z, \flat(z)\mto z$; it is then given by \[(\mu_{Z^{-1}})_y(\phi)=\int_{\cG^{\fr^\flat(\flat(z))}}\phi(\phi(g^{-1}\cdot \flat(z)))d\mu_\cG^{\fr^\flat(\flat(z))}(g), \ \rm{for} \ \phi\in \cC_c(Z^{-1}).\]
So, by using Notations~\ref{nota:inner-prod}, we get $\<\xi,\eta\>_{B_y}=\<\eta_{11},\zeta_{11}\>_{B_y}+\tRip<\eta_{21},\zeta_{21}>(y)$; hence $L^2(\cM_y;\cL)=L^2(\Ga;\cF)\oplus L^2(Z^{-1}_y;\overline{\cX})$. In the same way, we verify that $L^2(\cM_x;\cL)=L^2(Z_x;\cX)\oplus L^2(\cG_x;\cE)$. Thus, for all $\xi\in \cC_c(\cM;\cL)$, we have
\begin{eqnarray*}
	\|\xi\|_{\cstar_r(\cM;\cL)}=\text{max}\left\{\underset{y\in \Gamo}{\text{sup}}\|\pi_y^{\cM}(\xi)\|, \ \underset{x\in \Gpdo}{\text{sup}}\|\pi_x^{\cM}(\xi)\|\right\}.
\end{eqnarray*}
In particular, if $\xi=\left(\begin{array}{ll}\xi_{11}&0\\0&0\end{array}\right)\in \cC_c(\cM;\cL)$, and $y\in \Gamo$, then $\pi_y^\cM(\xi)=\pi_y^\Ga(\xi_{11})\oplus 0$, so that 
\begin{eqnarray}~\label{eq:norm_max}
\|\xi\|_{\cstar_r(\cM;\cL)}=\text{max}\left\{\|\xi_{11}\|_{\cstar_r(\Ga;\cF)}, \ \underset{x\in \Gpdo}{\text{sup}}\|\pi_x^\cM(\xi)\|\right\}.
\end{eqnarray}
Now, let $x\in \Gpdo$, and suppose $\eta\in \cC_c(\cM_x;\cL)$ is such that $\|\eta\|_{L^2(\cM_x;\cL)}\le 1$; i.e. $\text{max}\left\{\|\eta_{12}\|_{L^2(Z_x;\cX)}, \|\eta_{22}\|_{L^2(\cG;\cE)}\right\}\le 1$. Then, from a simple calculation we obtain
\[\<\pi_x^\cM(\xi)\eta,\pi_x^\cM(\xi)\eta\>_{A_x}=\Rip<\xi_{11}\cdot \eta_{12},\xi_{11}\cdot \eta_{12}>(x)=\Rip<R_x^\Ga(\xi_{11})\eta_{12},R_x^\Ga(\xi_{11})\eta_{12}>(x);\]
hence, by applying Proposition~\ref{pro:R_x}, we get $\|\pi_x^\cM(\xi)\eta\|_{L^2(\cM_x;\cL)}=\|R_x^\Ga(\xi_{11})\eta_{12}\|_{L^2(Z_x;\cX)}\le \|\xi_{11}\|_{\cstar_r(\Ga;\cF)}$. Therefore, from~\eqref{eq:norm_max}, we get $\|\xi\|_{\cstar_r(\cM;\cL)}=\|\xi_{11}\|_{\cstar_r(\Ga;\cF)}$.
\end{proof}

Recall (~\cite[p.14]{KMRW}) that if $\cA$ is a Dixmier-Douady bundle over $\cG$, and $\Ga\sim_Z\cG$, we define the \emph{pull-back} $\cA^Z$ over $\Ga$ as the quotient space of the pull-back $\fs^\ast \cA:=\{(z,a)\in Z\times\cA\mid \fs(z)=p(a)\}$ by $\cG$, where $\cG$ acts on $\fs^\ast \cA$ (on the right) by $(z,a)\cdot g:=(z\cdot g,\al_g^{-1}(a))$. 

\medskip

\begin{cor}~\label{cor:A_vs_A^Z}
Assume that $(\cA,\cG,\alpha)$ is a Dixmier-Douady bundle, and  that $\Ga\sim_Z\cG$. Then $\cA \rtimes_r\cG\sim_{Morita} \cA^Z \rtimes_r\Ga$.	
\end{cor}

\begin{proof}
Observe that for $\g\in \Ga$, the fibre $(s_\Ga^\ast \cA^Z)_\g=\cA^Z_{s_\Ga(\g)}$ is identified with $(Z_{s_\Ga(\g)}\times_{\Gpdo}\cA)/\cG$. Consider the $\cstar$-bundle $\fs^\ast \cA\To Z$. Then, the Fell system $(\Ga,s_\Ga^\ast\cA^Z)$ acts on $(Z,\fs^\ast \cA)$ on the left via
\begin{eqnarray}~\label{eq:Left-action-A-Z-A}
\frac{Z_{s_\Ga(\g)}\times_{\Gpdo}\cA}{\cG}\times \cA_{\fs(z)} \ni ([z,a],b)\mto ab \in \cA_{\fs(\g z)}=\cA_{\fs(z)},\end{eqnarray}
where $(\g,z)\in \Ga\ast Z$. Also, we have a right Fell $\cG$-pair $(s_\cG^\ast \cA,\fs^\ast\cA)$ over $Z$ determined by the right action 
\begin{eqnarray}~\label{eq:Right-action-A-Z-A}
\cA_{\fs(z)}\times \cA_{s_\cG(g)}\ni (a,b)\mto \alpha_g^{-1}(a)b\in \cA_{\fs(zg)}=\cA_{s_\cG(g)}.\end{eqnarray}
Next, define the inner products in the obvious way: if $(z,\flat(z'))\in Z\times_{\Gpdo} Z^{-1}$, we set 
\begin{eqnarray}~\label{eq:Ga-A^Z-inner-product}
\cA_{\fs(z)}\times \overline{\cA_{\fs(z)}} \ni (a,\flat(b))\mto [z,ab^\ast]\in (s_\Ga^\ast \cA^Z)_{{}_\Ga<z,z'>}=\frac{Z_{s_\Ga({}_\Ga <z,z'>)}\times_{\Gpdo}\cA}{\cG},
\end{eqnarray}
and if $(\flat(z),z')\in Z^{-1}\times_{\Gamo}Z$, we put 
\begin{eqnarray}~\label{eq:G-A-inner-product}
\overline{\cA_{\fs(z)}}\times \cA_{\fs(z')}\ni (\flat(a),b)\mto \alpha_{<z,z'>_\cG}^{-1}(a)^\ast b \in (s_\cG^\ast\cA)_{<z,z'>_\cG}=\cA_{s_\cG(<z,z'>_\cG)}=\cA_{\fs(z')}.
\end{eqnarray}
It is not hard to check that the settings (~\ref{eq:Left-action-A-Z-A}), (~\ref{eq:Right-action-A-Z-A}), (~\ref{eq:Ga-A^Z-inner-product}), and (~\ref{eq:G-A-inner-product}) give an equivalence of Fell systems $(\Ga,s_\Ga^\ast \cA^Z)\sim_{(Z,\fs^\ast \cA)}(\cG,s_\cG^\ast \cA)$. We thus complete the proof by applying Theorem~\ref{thm:Renault_equiv}. 	
\end{proof}

\medskip

\begin{rem}
In particular, it results from the last corollary that twisted $K$-theory (~\cite{TXL}) is invariant under Morita equivalences of locally compact Hausdorff groupoids; i.e. if $\cA\in \mathbf{Br}(\cG)$, and $\Ga\sim_Z\cG$, one has $K^\ast_\cA(\cG)\cong K^\ast_{\cA^Z}(\Ga)$.	
\end{rem}


\section{The reduced $\cstar$-algebra of an $\uc$-central extension}

Let $\cG$ be groupoid. Recall that (~\cite{TXL}, ~\cite{KMRW}, ~\cite{Tu}) an \emph{$\uc$-central extension} of $\cG$ is a pair $(\xymatrix{\uc \ar[r] & \wGa \ar[r]^{\pi} & \Ga}, P)$, where $\xymatrix{\uc \ar[r] & \wGa \ar[r]^{\pi} & \Ga}$ is a central groupoid extension, and $\Ga\sim_P\cG$; that is $\wGa^{(0)}=\Gamo$, $\uc$ acts continuously on $\wGa$, and $\pi:\wGa\To \Ga$ is an $\uc$-principal bundle. Such an object is symbolized as $(\wGa,P)$ if there is no risk of confusion. We say that $(\wGa_1,P_1)$ and $(\wGa_2,P_2)$ are Morita equivalent if there exists an $\uc$-equivariant Morita equivalence $Z: \wGa_1 \To \wGa_2$~\footnote{Let $\pi_i:\wGa_i\To \Ga_i, i=1,2$ be an $\uc$-principal bundle. A generalised morphism $Z:\wGa_1 \To \wGa_2$ is said to be $\uc$-equivariant if there is an action of $\uc$ on $Z$ such that $(\lambda \tilde{\g_1})\cdot z\cdot \tilde{\g_2}=\tilde{\g_1}\cdot(\lambda z)\cdot\tilde{\g_2}=\tilde{\g_1}\cdot z\cdot(\lambda \tilde{\g_2})$, for any $(\lambda,\tilde{\g_1},z,\tilde{\g_2}) \in \uc \times \wGa_1 \times Z \times \wGa_2$ such that these products make sense. Is is shown (~\cite{TXL}) that such a morphism induces a generalized morphism $Z/\uc:\Ga_1\To \Ga_2$.} such that the following diagrams commute (in terms of generalized morphisms):

\[
\xymatrix{\Ga_1 \ar[r]^{Z/\uc} \ar[dr]_{P_1} &  \Ga_2 \ar[d]^{P_2} \\ & \cG}
\ \ \text{and} \ \ 
\xymatrix{\Ga_1 \ar[r]^{Z/\uc} \ar[dr]_{Z_{\delta_1}} & \Ga_2 \ar[d]^{Z_{\delta_2}} \\ & \ZZ_2}
\]

The set of Morita equivalence classes of $\uc$-central extensions of $\cG$ is an Abelian group denoted by $\mathbf{Ext}(\cG,\uc)$. Note that the inverse of a class $[\EE]$ in $\mathbf{Ext}(\cG,\uc)$ is the class of the \emph{opposite} $\EE^{op}$ defined as follows: if $\EE=(\wGa,P)$, then $\EE^{op}:=(\wGa^{op},P)$, where $\wGa^{op}$ is $\wGa$ as a topological groupoid but the $\uc$-action is the conjugate one; i.e. $t\cdot \tilde{\g}^{op}:=(t^{-1}\tilde{\g})^{op}$. It is known (~\cite{KMRW}, ~\cite{Tu1},~\cite{M},~\cite{TXL}) that elements of $\mathbf{Ext}(\cG,\uc)$ are in bijection to those of of the $2$-cohomology group $\check{H}^2(\cG_\bullet,\uc)$; more precisely, there is an ismorphism of abelian groups $\mathbf{Ext}(\cG,\uc)\cong \check{H}^2(\cG_\bullet,\uc)$; we refer to ~\cite{KMRW}, ~\cite{Tu1}, ~\cite{M} for more details. \\
Given an $\uc$-central extension $\EE=(\wGa,P)$ of $\cG$, we form the Fell system $(L,\Ga)$ as follows: $$L:=\wGa\times_{\uc}\CC:=(\wGa\times\CC)/_{(\tilde{\g},t)\sim (\lambda\cdot \tilde{\g},\lambda^{-1}t), \lambda\in \uc}.$$
If $[\tilde{\g},t]$ denote the class of $(\tilde{\g},t)\in \wGa\times\CC$ in $L$, then we get a line bundle over $\Ga$ by setting $L\ni [\tilde{\g},t]\mto \pi(\tilde{\g})\in \Ga$. Next, we define the multiplication and the ${}^\ast$-involution on $L$ as $L_{\g_1}\times L_{\g_2}\ni ([\tilde{\g}_1,t_1],[\tilde{\g}_2,t_2])\mto [\tilde{\g}_1\tilde{\g}_2,t_1t_2]\in L_{\g_1\g_2}$, for $(\g_1,g_2)\in \Ga^{(2)}$, and $L_\g\ni [\tilde{\g},t]\mto [\tilde{\g}^{-1},t^{-1}]\in L_{\g^{-1}}$, respectively.\\

\begin{df}
The \texttt{reduced $\cstar$-algebra} $\cstar_r(\EE)$ of $\EE$ is defined as the reduced $\cstar$-algebra of the Fell system $(L,\Ga)$; i.e. $\cstar_r(\EE):=\cstar_r(\Ga;L)$. 	
\end{df}

\medskip

\begin{pro}(Compare with ~\cite[Proposition 3.3]{TXL}).
Suppose $\EE_1\sim \EE_2$; i.e. $[\EE_1]=[\EE_2]$ in $\mathbf{Ext}(\cG,\uc)$. Then $\cstar_r(\EE_1)\sim_{Morita}\cstar_r(\EE_2)$.	
\end{pro}
\begin{proof}
Suppose $\EE_i=(\xymatrix{\uc \ar[r] & \wGa_i \ar[r]^{\pi_i} & \Ga_i},\delta_i,P_i)$, and $L_i:=\wGa_1\times_{\uc}\CC$. If $Z:\wGa_1\To\wGa_2$ is an $\uc$-equivariant Morita equivalence, then take $\cX:=Z\times_{\uc}\CC=Z\times\CC/_{(z,t)\sim (\lambda z,\lambda^{-1}t)}$. Then, $\cX$ is a line bundle over $Z/\uc$, the projection being the map $[z,t]\mto [z]$, where $[z]$ is the class of $z$ in the quotient space $Z/\uc$. Furthermore, it is easy to verify that $(Z/\uc,\cX)$ implements an equivalence of Fell systems $\left(\Ga_1,L_1\right)\sim \left(\Ga_2,L_2 \right)$. Therefore, our assertion follows from Theorem~\ref{thm:Renault_equiv}.
\end{proof}

Let us now recall some constructions that we will need in the next result (see for instance ~\cite{KMRW,Tu,M} for more details). From an $\uc$-central extension $\EE=(\wGa,P)$ of $\cG$, one constructs a Dixmier-Douady bundle $(\cA_\EE,\cG,\al_\EE)$ in the following way. Let $\mu_{\wGa}=\{\mu_{\wGa}^y\}_{y\in \Gamo}$ be a Haar system on $\wGa$. For any $y\in \Gamo$, define the space $\cC_c(\wGa^y;\cH)^{\uc}$ of compactly supported continuous $\cH$-valued $\uc$-\emph{equivariant functions on} $\wGa^x$ as $$\cC_c(\wGa^y;\cH)^{\uc}:=\left\{\xi\in \cC_c(\wGa^y;\cH) \mid \xi(t\cdot \tilde{\g})=t^{-1}\xi(\tilde{\g}), \ \forall t\in \uc, \tilde{\g}\in \wGa^y\right\}.$$ Next, define a scalar product $\<\cdot,\cdot\>_y$ on $\cC_c(\wGa^y;\cH)^{\uc}$ by $\<\xi,\zeta\>_y:=\int_{\wGa^y}\<\xi(\tilde{\g}),\zeta(\tilde{\g})\>d\mu_{\wGa}^y(\tilde{\g})$. Denote by $\mathsf{H}^{\wGa}_y:=L^2(\wGa^y;\cH)^{\uc}$ the Hilbert space obtained by completing $\cC_c(\wGa^y;\cH)^{\uc}$ with respect to $\<\cdot,\cdot\>_y$, and let $\mathsf{H}^{\wGa}:=\coprod_{y\in \Gamo}\mathsf{H}^{\wGa}_y$. Then $\mathsf{H}^{\wGa}\To \Gamo$, being a countably generated continuous field of infinite-dimensional Hilbert spaces over the finite dimensional locally compact space $\Gamo$, is a locally trivial Hilbert bundle (cf. ~\cite[Th\'eor\`eme 5]{DD}). Moreover, $\wGa$ acts continuously and by unitaries on $\mathsf{H}^{\wGa}$ under the operation: $\cC_c(\wGa^{s(\g)};\cH)^{\uc}\ni \xi\mto \left(\tilde{\g}\cdot\xi:\wGa^{r(\g)}\ni \tilde{h}\mto \xi(\tilde{\g}^{-1}\tilde{h})\in \cH\right)\in \cC_c(\wGa^{r(\g)};\cH)^{\uc}$, for $\tilde{\g}\in \wGa$. Consider the continuous field of elementary $\cstar$-algebras $\cK^{\wGa}:=\coprod_{y\in \Gamo}\cK(\mathsf{H}^{\wGa}_y)$. Then $\cK^{\wGa}\To \Gamo$ is a locally trivial $\cstar$-bundle with fibre $\cK$, according to ~\cite[Th\'eor\`eme 8]{DD} (by comparing the field $\cH^{\wGa}$ with the trivial Hilbert bundle). Furthermore, there is a continuous action $\al$ of $\Ga$ by automorphisms on $\cK^{\wGa}$ given by: $\cK^{\wGa}_{s(\g)}\ni T\mto\tilde{\g}^{-1} T \tilde{\g}\in \cK^{\wGa}_{r(\g)}$, where $\tilde{\g}$ is any lift of $\g$ on $\wGa$, which gives us an element $(\cK^{\wGa},\Ga,\al)$ in $\mathbf{Br}(\Ga)$. Finally, we define $\cA_\EE$ over $\cG$ as the pull-back of $\cK^{\wGa}$ through the Morita equivalence $\cG\sim_{P^{-1}}\Ga$; i.e. $\cA_\EE:=(\cK^{\wGa})^{P^{-1}}$. This construction gives a homomorphism of abelian groups $\Psi:\mathbf{Ext}(\cG,\uc)\To \mathbf{Br}(\cG)$. Conversely, from a Dixmier-Douady bundle over $\cG$, it is not hard to build an $\uc$-central extension of $\cG$, and then to construct a homomorphism $\mathbf{Br}(\cG)\To \mathbf{Ext}(\cG,\uc)$ which is inverse to $\Psi$ (~\cite{KMRW}, ~\cite{TXL}). 

\begin{thm}
Let $\EE \in \mathbf{Ext}(\cG,\uc)$, and let $(\cA_\EE,\cG,\alpha_\EE)$ in $\mathbf{Br}(\cG)$ be its corresponding Dixmier-Douady bundle over $\cG$. Then, under the above constructions and notations, we have 
\[\cA_\EE\rtimes_r\cG \sim_{Morita} \cstar_r(\EE^{op}).\]
\end{thm}

\begin{proof}
Write $\EE=(\xymatrix{\uc \ar[r] & \wGa \ar[r]^\pi & \Ga}, Z)$, where $Z:\cG\To \Ga$ is a Morita equivalence. For the sake of simplicity, we will denote $\cA:=\cK^{\wGa}$. From Corollary~\ref{cor:A_vs_A^Z}, we have 
\[\cA_\EE\rtimes_r\cG \sim_{Morita} \cA\rtimes_r\Ga:=\cstar_r(\Ga;s^\ast \cK^{\wGa}). \]
Thus, we only have to show that 
\begin{eqnarray}
\cstar_r(\Ga;s^\ast \cA)\sim_{Morita} \cstar_r(\Ga;L)=:\cstar_r(\EE^{op}),  \ \text{where} \ L:=\wGa^{op}\times_{\uc}\CC.
\end{eqnarray}
However, again in view of the Renault's equivalence Theorem~\ref{thm:Renault_equiv}, it suffices to build an equivalence between the Fell systems $(\Ga,s^\ast\cA)$ and $(\Ga,L)$. \\
Consider the Banach bundle $\cX:=s^\ast \sfH^{\wGa}$ over $\Ga$ defined as the pull-back of the Hilbert $\wGa$-bundle $\sfH^{\wGa}\To \Gamo$ through the source map of $\Ga$. We claim that $\cX$ implements the desired equivalence over $\Ga$; that is, that  
\begin{eqnarray}
(\Ga,s^\ast\cA)\sim_{(\Ga,\cX)}(\Ga,L). 
\end{eqnarray}
From the $\wGa$-action on $\sfH^{\wGa}$ defined in the discussion before the theorem, we get a left action of $s^\ast\cA$ on $\cX$ given by
\begin{eqnarray}~\label{eq:action:K(H^Ga)_H^Ga}
\begin{array}{lll}
	\cK(L^2(\wGa^{s(\g_1)},\cH)^{\uc})\times L^2(\wGa^{s(\g_2)},\cH)^{\uc} & \To & L^2(\wGa^{s(\g_1\g_2)},\cH)^{\uc} \\
	\ \ \ \ \ \ \ \ \ \ \ \ \ \ \ \ \ (T \ \ \ \ \ \ \ \ , \ \ \ \ \ \ \ \ \xi)  & \mto & T\cdot \xi:=\tilde{\g}_2^{-1}T(\tilde{\g}_2\xi), \ \ \ 
\end{array}
\end{eqnarray}
 and a right action of $L$ on $\cX$ 
\begin{eqnarray}~\label{eq:action:H^Ga_L}
	\begin{array}{lll}
		L^2(\wGa^{s(\g_2)},\cH)^{\uc}\times \left(\wGa_{\g_3}^{op}\times_{\uc}\CC\right) & \To & L^2(\wGa^{s(\g_2\g_3)},\cH)^{\uc} \\
		\ \ \ \ \ \ \ \ \ \ \ \ \ \ \ \ \ (\xi \ \ \ \ , \ \ \ \ [\tilde{\g},\lambda]) & \mto & \xi\cdot [\tilde{\g},\lambda]:= \tilde{\g}_3^{-1}\cdot \lambda\xi
	\end{array}
\end{eqnarray}
where $\tilde{\g}_3$ is any lift of $\g_3$ in $\wGa$. The maps (~\ref{eq:action:K(H^Ga)_H^Ga}) and (~\ref{eq:action:H^Ga_L}) are continuous since the $\Ga$-actions are continuous. Also, they are full since the actions are, in fact, isomorphisms.\\
We now construct the $s^\ast\cA$-valued and $L$-valued inner products $\cX\ast \overline{\cX}\To s^\ast\cA_{{}_\Ga<}$ and $\overline{\cX}\times \cX\To L_{>_\Ga}$, respectively. Note that, as in Example~\ref{ex:reflexivity}, $\Ga^{-1}=\{\g^{-1}\mid \g\in \Ga\}$,  if $(\g,\flat(\g'))\in \Ga\times_{\Gamo}\Ga^{-1}$ (in other words, $s(\g)=s(\g')$), then ${}_{\Ga}<\g,\g'>=\g\g'^{-1}$, and if $(\flat(\g'),\g")\in \Ga^{-1}\times_{\Gamo}\Ga$ (i.e. $r(\g')=r(\g")$), then $<\g',\g">_{\Ga}=\g'^{-1}\g"$. We then define these inner products as 
\begin{eqnarray}~\label{eq:scalar-product:A_Ga<}
\begin{array}{lll}
    \cX_\g\times \overline{\cX}_{\g'} & \To & \cA_{s(\g\g'^{-1})}=\cK(L^2(\wGa^{r(\g')},\cH)^{\uc})\\
    (\xi,\flat(\eta)) & \mto & {}_{s^\ast \cA}\<\xi,\eta\>:=\theta_{\tilde{\g}'\xi,\tilde{\g}'\eta}	
\end{array}
\end{eqnarray}
where for $\zeta,\zeta'\in L^2(\wGa^y,\cH)^{\uc}$, $\theta_{\zeta,\zeta'}\in \cK(L^2(\wGa^y,\cH)^{\uc})$ is the rank one operator $$L^2(\wGa^y,\cH)^{\uc}\ni \zeta"\mto (\<\zeta',\zeta"\>_x)\zeta \in L^2(\wGa^y,\cH)^{\uc}, \ y\in \Gamo;$$ and 
\begin{eqnarray}~\label{eq:scalar-product:L_>_Ga}
	\begin{array}{lll}
	\overline{\cX}_{\g'}\times \cX_{\g"} & \To & L_{\g'^{-1}\g"}=\wGa_{\g'^{-1}\g"}^{op}\times_{\uc}\CC\\
	(\flat(\xi),\eta) & \mto & \<\xi,\eta\>_L:=\left[\tilde{\g'}^{-1}\tilde{\g}", \<\tilde{\g}'\xi,\tilde{\g}"\eta\>_{r(\g')} \right]		
	\end{array}
\end{eqnarray}
where, as usual, $\tilde{\g'}$ and $\tilde{\g}"$ are any lifts of $\g'$ and $\g"$, respectively. Recall that for $y\in \Gamo$, the scalar product $\<\cdot,\cdot\>(y)$ on $\sfH^{\wGa}_y=\cX_y$ is defined as \[\<\xi,\eta\>(y)=\int_{\wGa^y}\<\xi(\tilde{h}),\eta(\tilde{h})\>d\mu_{\wGa}^y(\tilde{h}) \in \CC.\]
The algebraic properties of these maps are easy to check. The map (~\ref{eq:scalar-product:A_Ga<}) is full, for $\text{span}\left\{\theta_{\zeta,\zeta'} \mid \zeta, \zeta'\in L^2(\wGa^{r(\g')},\cH)^{\uc}\right\}$ is the ideal of finite-rank operators on $L^2(\wGa^{r(\g')},\cH)^{\uc}$ and the map $L^2(\wGa^{s(\g')},\cH)^{\uc}\To L^2(\wGa^{r(\g')},\cH)^{\uc}$ given by the $\wGa$-action is an isomorphism of Hilbert spaces. The map~\eqref{eq:scalar-product:L_>_Ga} is clearly surjective. Thus, it only remains to verify that the compatibility condition (cf. Definition~\ref{df:equiv-Fell} (ii)) holds; that is, for any triple $(\g,\g'^{-1},\g")\in \Ga\times_{\Gamo}\Ga^{-1}\times_{\Gamo}\Ga$,
\begin{eqnarray}
\xi\cdot\<\xi',\xi"\>_L={}_{s^\ast\cA}\<\xi,\xi'\>\cdot\xi", \ \forall (\xi,\flat(\xi'),\xi")\in \cX_\g\times\overline{\cX}_{\g'}\times\cX_{\g"}.\end{eqnarray} 
One has
\begin{align*}
\xi\cdot \<\xi',\xi"\>_L & = \xi\cdot\left[\tilde{\g}'^{-1}\tilde{\g}",\<\tilde{\g}'\xi',\tilde{\g}"\xi"\>({r(\g')})\right] \\
 & =\tilde{\g}"^{-1}\tilde{\g}'\cdot (\<\tilde{\g}'\xi',\tilde{\g}"\xi"\>({r(\g')}))\xi\\
 & = \tilde{\g}"^{-1}\cdot(\<\tilde{\g}'\xi',\tilde{\g}"\xi"\>({r(\g')}))(\tilde{\g}'\xi) \\
 & = \tilde{\g}"^{-1}\cdot \theta_{\tilde{\g}'\xi,\tilde{\g}'\xi'}(\tilde{\g}"\xi") \\
 &=  {}_{s^\ast\cA}\<\xi,\xi'\>\cdot\xi",	
\end{align*}
which completes the proof.
\end{proof}


\addcontentsline{toc}{chapter}{\bibname}
\bibliography{}

\end{document}